\documentclass[pdflatex,sn-mathphys-num]{sn-jnl}

\usepackage{multirow,booktabs}
\usepackage{amsmath,amssymb,amscd,amsthm,amsxtra, esint}

\usepackage{mathrsfs}

\usepackage{color} 
\usepackage{ulem}
\usepackage[makeroom]{cancel}

\allowdisplaybreaks[2]

\definecolor{gr}{rgb}   {0.,   0.69,   0.23 }
\definecolor{bl}{rgb}   {0.,   0.5,   1. }
\definecolor{mg}{rgb}   {0.85,  0.,    0.85}
\definecolor{yl}{rgb}   {0.8,  0.7,   0.}
\definecolor{or}{rgb}  {0.7,0.2,0.2}

\newtheorem{theorem}{Theorem} [section]

\newtheorem{lemma}[theorem]{Lemma}
\newtheorem{proposition}[theorem]{Proposition}

\newtheorem{definition}[theorem]{Definition}
\newtheorem{corollary}[theorem]{Corollary}

\newtheorem*{ackno}{Acknowledgements}

\newcommand{\noi}{\noindent}
\newcommand{\Z}{\mathbb{Z}}
\newcommand{\R}{\mathbb{R}}

\newcommand{\T}{\mathbb{T}}

\let\Re=\undefined\DeclareMathOperator*{\Re}{Re}
\let\Im=\undefined\DeclareMathOperator*{\Im}{Im}

\let\P= \undefined
\newcommand{\P}{\mathbf{P}}

\newcommand{\E}{\mathbb{E}}

\renewcommand{\L}{\mathcal{L}}

\newcommand{\al}{\alpha}

\newcommand{\Dl}{\Delta}
\newcommand{\eps}{\varepsilon}

\newcommand{\ld}{\lambda}

\newcommand{\s}{\sigma}

\newcommand{\ft}{\widehat}

\newcommand{\dt}{\partial_t}

\renewcommand{\o}{\omega}
\renewcommand{\O}{\Omega}

\newcommand{\les}{\lesssim}

\newcommand{\jb}[1]
{\langle #1 \rangle}

\newcommand{\ind}{\mathbf 1}

\newcommand{\N}{\mathbb{N}}

\numberwithin{equation}{section}
\numberwithin{theorem}{section}

\newcommand{\PP}{\mathbb{P}}

\newcommand{\C}{\mathbb{C}}

\DeclareMathOperator{\Law}{Law}

\newcommand{\dr}{\theta}
\newcommand{\Dr}{\Theta}

\newcommand{\Ha}{\mathbb{H}_a}

\newcommand{\proj}{\Pi}

\begin{document}

\title[Focusing Gibbs measures with harmonic potential]
{Phase transition for weakly interacting focusing Gibbs measures with harmonic potential}

\author*[1]{\fnm{Damiano} \sur{Greco}}

\author[2]{\fnm{Yuzhao} \sur{Wang}}

\affil*[1]{\orgdiv{School of Mathematics}, \orgname{The University of Edinburgh, and  The Maxwell Institute for the Mathematical Sciences, James Clerk Maxwell Building}, \orgaddress{\street{Peter Guthrie Tait Road}, \city{Edinburgh}, \postcode{EH9 3FD}, \state{United Kingdom}}} 

\affil[2]{\orgdiv{School of Mathematical Sciences}, \orgname{Dalian University of Technology}, \orgaddress{\street{Street}, \city{ Dalian}, \postcode{ 116024}, \state{China}; \orgdiv{School of Mathematics}, \orgname{University of Birmingham, Watson Building},  \orgaddress{\street{Edgbaston}, \city{Birmingham}, \postcode{B15 2TT}, \state{United Kingdom}}}}

\email{dgreco@ed.ac.uk}
\email{y.wang.14@bham.ac.uk}

\abstract{In this paper,  we study the Gibbs measures on Euclidean spaces associated to the focusing nonlinear
Schr\"odinger equation with harmonic potential and critical non linearity whose coupling constant tends to 0, a question initially posed by Brydges-Slade (1996) for the {\unboldmath {$\Phi^4_2$}}-model on {\unboldmath{$\mathbb{T}^2$}}.  In dimension one and
in the higher dimensional cases (with radial assumption), we establish a critical threshold below which  the frequency-truncated measures converge to the base Gaussian measure (possibly with a renormalized {\unboldmath{$L^2$}} cut-off) while,
in the supercritical regime, we prove non-convergence of the frequency-truncated measures, even up to a subsequence.
}




\maketitle

\section{Introduction}\label{sec1}

In this paper, we study the Gibbs measure  $\rho$ on $\mathcal{S}'(\R^d)$,  the space of tempered distributions on $\R^d$,  formally given by 
\begin{align}
\label{gibbs_intro}
d\rho(u)=\mathcal{Z}^{-1}\exp\big(\!-H(u)\big)du, 
\end{align}
where $H$ denotes the Hamiltonian functional 
\begin{align}
\label{ham}
H(u)=\frac{1}{2}\int_{\R^d}|\nabla u|^2dx+\frac{1}{2}\int_{\R^d}|x|^2|u|^2dx-\frac{\ld}{p}\int_{\R^d}|u|^pdx,
\end{align}
with coupling $\ld\in \R$ representing the strength of the  interaction, $p>2$ being a real number,  and  $\mathcal{Z}$ denoting the normalizing constant (partition function)\footnote{Hereafter, we use $\mathcal{Z},Z$  to denote various normalization constants whose values may change line by line.}.  In view of the definition of the Hamiltonian $H$ in \eqref{ham}, as already discussed in \cite{RSTY}, the first attempt to properly define the measure $\rho$ consists in interpreting \eqref{gibbs_intro} as 
\begin{align}
\label{gibbs_introbis}
d\rho(u)=\mathcal{Z}^{-1}\exp\bigg(\frac{\ld}{p}\int_{\R^d}|u|^pdx\bigg)d\mu(u),
\end{align}  
where $\mu$ denotes the Gaussian free field on $\R^d$ formally given by 
\begin{align}
\label{Gaussian}
d\mu(u)={\mathcal{Z}}^{-1}\exp\bigg(-\frac{1}{2}\int_{\R^d}\Big(|\nabla u|^{2}+|x|^2|u|^2\Big)dx\bigg)du.
\end{align}
\noi As a result, one can try to show that $\rho$ in \eqref{gibbs_introbis} defines a probability measure with density (with respect to the Gaussian measure $\mu$ in \eqref{Gaussian}) given by the exponential of the interaction potential $\frac{\ld}{p}\int_{\R^d}|u|^pdx$. 
We point out that the problem of constructing Gibbs measures as in \eqref{gibbs_intro} has attracted many authors in the recent years in connection with study of the dynamics of the nonlinear Schr\"odinger (NLS) with harmonic potential 
\begin{align}
\label{nls}
i\partial_tu=(-\Delta+|x|^2)u-\ld|u|^{p-2}u, \quad (t,x)\in \R^{1+d},
\end{align}
see e.g.,  \cite{BO96,BTT, D1,D3,D2,RSTY} and references therein for further details on the topic. For the convenience of the reader, before considering the constructibility issue for the Gibbs measure with harmonic potential in \eqref{gibbs_intro}, we provide an overview concerning the non-harmonic case in the periodic setting, namely on $\T=\R/(2\pi\Z)\simeq [0,2\pi]$. 

The seminal work of Lebowitz, Rose, Speer \cite{LRS} 
followed by Bourgain \cite{BO94}, 
initiated the study of the focusing Gibbs measures  associated to the 
nonlinear Schr\"odinger equation (NLS) on $\T$:
\begin{align}
i \dt u +  \Dl u + |u|^{p-2} u = 0.
\label{NLS1}
\end{align}
Namely, they considered the measure $\rho$ 
formally defined as:\footnote{Here, we ignore the issue at the zeroth frequency.}
\begin{align}
d\rho(u) = {{Z}}^{-1} \exp\bigg(\frac 1 p\int_{\T} |u|^p d x\bigg)d{\mu}(u),
\label{Gibbs1a}
\end{align}
with
 \begin{align}
 \label{muNLS}
 d{\mu}(u)={{Z}}^{-1}\exp\bigg(-\frac{1}{2}\int_{\T}|\partial_x u|^{2}dx\bigg)du.
\end{align}

From the point of view of the (non-)equilibrium statistical mechanics, the equation \eqref{NLS1} (as well as the higher dimensional cases) has received extensive attention;
 see 
 \cite{BO97,Tzv1,  Tzv2, OQV, LMW, BBulut, 
C1, D3, FOW}.

The density in \eqref{Gibbs1a}, because of its super-Gaussian growth, 
is never integrable with respect to the Gaussian measure ${\mu}$ in \eqref{muNLS}.
Namely, the Gibbs measure $\rho$ in \eqref{Gibbs1a}
is not normalizable to be a probability measure. In order to overcome this issue,  since the mass is a conserved quantity under the flow of the periodic NLS \eqref{NLS1}, 
Lebowitz, Rose, and Speer \cite{LRS} proposed to consider the following Gibbs measure  $\rho=\rho_{p,K}$ with an  $L^2$ cut-off:
\begin{align}
\label{GibbsNLS}
d\rho(u) ={{Z}}_{p,K}^{-1} \ind_{\{\|u\|_{L^2(\T)}\le K\}}\exp\bigg(\frac 1 p\int_{\T} |u|^p d x\bigg)d{\mu}(u),
\end{align}
  and proved  that the normalizability of $\rho$ in \eqref{GibbsNLS} depends on $p$ being $L^2$-subcritical, supercritical, or critical. Namely, ${Z}_{p,K}<\infty$ if $2<p<6$ and  $K>0$, or  $p=6$ and $K<\|Q\|_{L^2(\R)}$, while  ${Z}_{p,K}=\infty$ if $p=6$ and $K>\|Q\|_{L^2(\R)}$, or $p>6$ and $K>0$, 
where $Q$ denotes the unique\footnote{Up to symmetries.} optimizer for the Gagliardo-Nirenberg-Sobolev (GNS) inequality on $\R$ (see Proposition \ref{GNS_prop}); see also \cite{LW22} for an analogous picture in the framework of the fractional NLS. 
More recently, in \cite{OST2}, normalizability has been established at mass critical case $p=6$ with critical  cut-off size $K=\|Q\|_{L^2(\R)}$, thus completing the constructibility picture for the Gibbs measures in \eqref{GibbsNLS} and exhibiting  a phase transition phenomena  with respect to $K$ when $p=6$. 
We also refer the reader to \cite{GLLOW} for an analysis on  the sharp divergence rate of the partition function in the non-normalizability regimes described above. 

Next, we move our attention to the Gibbs measure with harmonic potential in \eqref{gibbs_introbis}. Note that, so far, we have essentially only discussed the construction of the Gibbs measures on the torus, and moving to the euclidean setting is a highly non-trivial task; see \cite{BG2, Baras, GH}.
In particular, besides the physical relevance of \eqref{nls} to describe Bose-Einstein condensates \cite{BBDBG,DS}, adding the confining potential $|x|^2$  to the Hamiltonian \eqref{ham} allows us to work directly without infrared cut-off in constructing the Gibbs measure \eqref{gibbs_introbis}. 

The problem of constructing \eqref{gibbs_introbis} arises from a work of Burq-Thomann-Tzvetkov \cite{BTT}, if $d=1$. When $\ld=-1$ (i.e., the defocusing case), it is not difficult to see that  the potential energy $\frac{\ld}{p}\int_{\R^d}|u|^pdx$ is finite $\mu$-a.surely (see e.g., \cite[Corollary 2.4]{RSTY}) and, more importantly, it is exponentially integrable with respect to $\mu$ in \eqref{Gaussian}. As a consequence, in the defocusing case  there is no issue in making sense of \eqref{gibbs_introbis}.
As a result, in the sequel we consider the focusing case  $\ld=1$. In analogy  with the focusing NLS case \eqref{Gibbs1a}, the potential term is not exponentially integrable anymore and \eqref{gibbs_introbis} is replaced by
\begin{align}
\label{GIbbs_def}
d\rho(u)=\mathcal{Z}_{d,p,K}^{-1}\ind_{\left\{|\int_{\R^d}:|u|^2:dx|\le K\right\}}\exp\bigg({\frac{\ld}{p}\int_{\R^d}}|u|^pdx\bigg)d\mu(u), 
\end{align} 
where $K$ is a positive parameter, and $\int_{\R^d}:\!\!|u|^2\!\!:dx$ denotes the \textit{Wick} renormalized $L^2$-norm defined in \eqref{wick}.  The reason for introducing such a renormalization is due to the fact that $\int_{\R^d}|u|^2dx=\infty$ $\mu$-a.surely; see \eqref{bup}.  As a result, a renormalization is needed to make sense of the cut-off in the density. We refer the reader for example to \cite{GOTT, OST} for a similar renormalization issue. 
\noi
If  $d=1$ and $p=4$ the measure in \eqref{GIbbs_def}  was first constructed by Burq, Thomann, and Tzvetkov \cite{BTT}. Later, in \cite{RSTY}, the authors completed the picture of normalizability/non-normalizability by proving that $\mathcal{Z}_{1,p,K}<\infty$ if $2<p<6$ \eqref{GIbbs_def}, while  $\mathcal{Z}_{1,p,K}=\infty$ if $p\ge 6$,  resulting in a similar picture to the one for \eqref{GibbsNLS}. However, we highlight that in the harmonic potential case the phase transition always occurs at the critical $L^2$ exponent  regardless of the size of the cut-off $K>0$. 
\noi
A similar result holds for $d\ge 2$ under radial assumptions\footnote{To be precise, it holds by defining the measure $\mu$ in \eqref{Gaussian} on $\mathcal{S}'_{{\text{rad}}}(\R^d)$, the space of radial Schwartz distributions. Moreover, an extra restriction on $p$ is required if  $d\ge 3$.}, with critical index given by $p^*(d)=2+\frac{4}{d}$. It is worth to mention that, prior to the contribution of \cite{RSTY}, in the case $d=2$ and $\ld=1$, Deng \cite{D1} proved normalizability of \eqref{GIbbs_def}  if $2<p<4$, again under radial assumptions.

Finally, let us consider the generic case $\ld>0$. Historically, in \cite[p.\,489]{BS}, Brydges and Slade proposed to study the so called $\Phi^{4}_2$-measure by taking into account the strength of the coupling interaction $\ld>0$. To be precise, because of the non-normalizability of the $\Phi^{4}_2$-measure (or in general of the $\Phi^{4}_{d}$-measure, see \cite{OST})  for any fixed $\ld>0$, they further questioned (1996) about the existence of a critical behavior of $\ld=\ld_N$, $N\in \N$, such that the frequency truncated $\Phi^{4}_{2}$-measure (up to the frequency $N$)\footnote{With a $L^2$ cut-off $K=K_N$.} converges to a measure. In particular, they were interested in detecting a ``critical point, separating the weak and strong coupling regimes'', In \cite{GOTT}, by adapting the analysis carried out in \cite{OST}, the authors proved the existence of such a critical point, complementing the critical models in $d=1,3$ studied in \cite{OST,OOT,OOT2}.  The purpose of this paper is to carry out a similar analysis to the one of \cite{GOTT}  for the Gibbs measure with harmonic potential in \eqref{GIbbs_def}.  Our main result indeed addresses the case of the $L^{2}$-critical nonlinearity $p^*(d)=2+\frac{4}{d}$ showing that 
$\ld_N\sim (K_N+\log N)^{-\frac{2}{d}}$ separates  the weak and the strong coupling regime; see Theorem \ref{main} for a precise statement. 


\subsection{Harmonic operator}
Before stating our main contributions, we recall the definition of the Harmonic oscillator which will play a role in the sequel.  
The operator $\mathcal{L}$ associated with \eqref{nls} is defined by
\begin{align*}
\mathcal{L}=-\Delta +|x|^2,
\end{align*}
where $\Delta=\sum_{i=1}^d\partial^2_{x_i}$ denotes the usual Laplace operator on $\R^d$. 
It is known that the operator $\mathcal{L}$ admits a self-adjoint extension on $L^{2}(\R)$ if $d=1,$ respectively $L^{2}_{\text{rad}}(\R^d)$ if $d\ge 2$, with eigenfunctions $\left\{h_n\right\}_{n\in \N}$
\begin{align}\label{EIGEN}
h_n(x):=\begin{cases}
c_n e^{-\frac{x^2}{2}} H_n(x), & \mbox{if }d=1, \\ c(d)\frac{\sqrt{n!}}{\sqrt{\Gamma(n+\frac{d}{2})}}L^{(\frac{d}{2}-1)}_{n}(|x|^2)e^{-\frac{|x|^2}{2}}, & \mbox{if }d\ge 2,
\end{cases}
\end{align}
where $L^{(\alpha)}_n$ denotes the Laguerre polynomial of type $\alpha$ and degree $n$,  $c_n:=(n!)^{-\frac{1}{2}}2^{-\frac{n}{2}}\pi^{-\frac{1}{4}}$, $c(d):=\frac{\sqrt{2}}{\sqrt{\text{Vol}(\mathcal{S}^{d-1})}}$,  and $H_n(x):=(-1)^n \frac{d^n}{dx^n}(e^{-x^2})$ is the $n$-th Hermite polynomial;  see also \cite{RDL} for the definitions of the Laguerre polynomials.
The corresponding eigenvalues $\left\{\nu^2_n\right\}_{n\in \N}$ are
 \begin{align}\label{ld_n}
\nu^2_n=\begin{cases} {1+2n}, & \mbox{if }d=1,\\ {d+4n}, & \mbox{if }d\ge 2.
\end{cases}
\end{align}
We also recall  the uniform bound  on the $n$--th Hermite function $h_{n}(x)$ (see \cite[eq. (2.2)]{B1}):
\begin{align}\label{bound_H}
\|h_{n}\|_{L^{\infty}(\R)}\le \pi^{-\frac{1}{4}}, 
\end{align}
as well as the following identity which follows by combining the addition formula for the Laguerre polynomials (see e.g.,  \cite{K}) with \cite{C}
\begin{align}\label{L_formula}
L^{(\frac{d}{2}-1)}_{n}(|x|^2)=\frac{(-1)^n}{4^n n!}\sum_{m_1+\dots +m_d=n}\prod_{i=1}^{d}H_{2m_i}(x_i).
\end{align}
Associated with the Harmonic operator $\mathcal{L}$, we define the so-called Harmonic  Sobolev spaces as follows:


\begin{definition} Let $1 \le p\le \infty$ and $s\in \R$, we define the harmonic Sobolev space $\mathcal{W}^{s,p}(\R)$ (respectively $\mathcal{W}_{\text{rad}}^{s,p}(\R^d)$ if $d\ge 2$) the space of functions (respectively radial functions) induced by the 
 norm
\begin{align*}
\|u\|_{\mathcal{W}^{s,p}(\R^d)}:=\|\mathcal{L}^{\frac{s}{2}}u\|_{L^{p}(\R^d)}.
\end{align*}
If $p=2$ we set $\mathcal{H}^s:=\mathcal{W}^{s,2}$, and for  $u$ of the form $u=\sum_{n\ge 0} c_n h_n$ we have 
\begin{align}\label{H1_curl}
\|u\|^2_{\mathcal{H}^s(\R^d)}:=\|\mathcal{L}^{\frac{s}{2}}u\|^2_{L^{2}(\R^d)}=\sum_{n\ge 0} \nu^{2s}_n |c_n|^{2}. 
\end{align}
\end{definition}

Finally, we recall the Gagliardo-Nirenberg-Sobolev inequality on $\R^d$ with $d\ge 1$. As far as concerns existence of optimizers for such inequality, we refer the reader to \cite{Nagy} for the one dimensional case, and to \cite{Weinstein} for $d\ge 2$.
In general,  we refer for example to \cite[Proposition 3.1]{Frank} and \cite[Theorem 2.1]{BFV14} and references therein for more general Gagliardo-Nirenberg-Sobolev inequalities. 
\begin{proposition}\label{GNS_prop}
Let $p > 2$ if $d=1,2 $, or  $2 < p < \frac{2 d}{d - 2 }$ if $d \ge 3$. Then, 
\begin{align}
\|u\|^p_{L^p(\R^d)}\le \textup{C}_{\textup{GNS}}(p,d)\|u\|^{\frac{(p-2)d}{2}}_{\dot{H}^{1}(\R^d)}\|u\|^{2+\frac{p-2}{2}(2-d)}_{L^2(\R^d)}. 
\label{GNS_0}
\end{align}
Moreover, the optimal constant $\textup{C}_{\textup{GNS}}$ is achieved by a function in $H^1(\R^d)$. 
\end{proposition}

In particular, by \eqref{GNS_0} and \eqref{H1_curl} we infer
\begin{align}\label{GNS}
\|u\|^p_{L^p(\R^d)}\le\textup{C}_{\textup{GNS}}(p,d) \|u\|^{\frac{(p-2)d}{2}}_{\mathcal{H}^{1}(\R^d)}\|u\|^{2+\frac{p-2}{2}(2-d)}_{L^2(\R^d)}. 
\end{align}

\subsection{Measure construction and main result}
In this subsection, we give more details on the definition of the Gibbs measure \eqref{GIbbs_def} if $d=1$, and if $d\ge 2$ with radial assumptions. Furthermore, we state our main results concerning normalizability in the weakly coupling regime and non-normalizability in the strong coupling regime. 
For simplicity of the presentation, we focus on the case $d\ge 2$, with the understanding that the same holds for $d=1$ without any radial assumption.

By using the eigenbasis $\left\{h_n\right\}_{n\in \N}$, any distribution $u\in \mathcal{S}'_{\text{rad}}(\R^d)$ can be written  as follows
\begin{align*}
u=\sum_{n\ge 0}u_n h_n, \quad u_n=\langle u,h_n\rangle.  
\end{align*}
Next, given $N\in \N$, we define the spectral projector $\P_{N}$ by 
\begin{align}
\label{proj}
\P_{N}u=\P_{N}\Big(\sum_{n\ge 0}u_n h_n\Big):=\sum_{n=0}^{N}u_n h_n. 
\end{align}
As already explained in \cite[Section 1.2]{RSTY},  the Gaussian measure in \eqref{Gaussian} can be properly defined as the induced probability measure under the map
\begin{align}
\o\in \O \longmapsto u^{\o}= \sum_{n=0 }^{\infty} \frac{ g_n(\o)}{\nu_n} h_n, 
\label{map}
\end{align}
where $\left\{g_n \right\}_{n\in \N}$ is a sequence of independent standard complex-valued Gaussian random variables
on a probability space $(\Omega, \mathcal{F}, \PP)$.\footnote{In particular,
$\Re g_n$ and $\Im g_n$ are real-valued Gaussian random variables
with mean 0 and variance $\frac 12$.}
Note that,  if $u^{\o}$ is as in \eqref{map}, $\P_{N}u^{\o}$ is a mean zero, complex valued Gaussian random variable with variance
\begin{align*}
\s_N=\E\big[|\P_{N}u^{\o}|^2\big]=\sum_{n=0}^{N}\frac{h^{2}_n(x)}{\nu^2_n}, 
\end{align*}
from which 
\begin{align}\label{bup}
\E\big[\|\P_{N}u^{\o}\|^2_{L^{2}(\R^d)}\big]=\int_{\R^d}\s_N(x)dx=\sum_{n=0}^{N}\frac{1}{\nu^2_n}\sim \log N\to +\infty,\quad  \text{as}\ N\to +\infty.
\end{align}
As a result, a typical  function in the support of $\mu$ fails to be square integrable.
Indeed, one can prove that $\{\P_N u^{\o}\}_N$ is cauchy in  $L^{2}(\Omega, \mathcal{H}_{\text{rad}}^{-s}(\R^d))$, for any $s>0$.  
The above analysis,  motivates us to introduce the Wick renormalized power  $:\!|\P_{N}u|^2\!:$ via
\begin{align*}
:\!|\P_{N}u|^2\!: \stackrel{\text{def}}{=}|\P_{N}u|^2-\s_N. 
\end{align*}
It is also well known, cf., \cite[Lemma 3.6]{BTT}, that $\int_{\R^d}\!\!:\!\!|\P_{N}u(x)|^2\!\!:\!dx$ is Cauchy sequence in $L^2(\mathcal{H}_{\text{rad}}^{-s}(\R^d), d\mu)$ for any $s>0$. In particular, we can define  
\begin{align}\label{wick}
\int_{\R^d}:\!|u(x)|^2\!:dx \stackrel{\text{def}}{=}\lim_{N\to +\infty}\int_{\R^d}:\!|\P_{N}u(x)|^2\!:dx.
\end{align}
 On the other hand, by Corollary \ref{Lp_conv}, one can see that $u^\o$ in \eqref{map} belongs to $L^{p}(\R^d)$ a.s.,  provided $p$ satisfies the same assumptions of Proposition \ref{GNS_prop}. Hence, no renormalization is needed for the potential term. Then, if we define
\begin{align}
\label{Wick0}
R_{\,p,N}(u):=\frac{1}{p}\int_{\R^d}|\P_{N}u|^{p}dx, 
\end{align}
and $\{\ld_N\}_{N\in \N}$ is a sequence of positive numbers, we consider the truncated Gibbs measure $\rho_{p,N}$:
\begin{align}
\label{trunc_gibbs}
d\rho_{p,N}(u)=\mathcal{Z}^{-1}_{p,N}\ind_{\left\{|\int_{\R^d}:|\P_{N}u|^2:dx|\le K_N\right\}}\exp\big(\ld_N R_{\, p,N}(u)\big)d\mu(u),
\end{align}
where  we formally wrote $\mu$ by 
\begin{align}
\label{mu_def}
d\mu={\mathcal{Z}}^{-1}\exp\Big(-\frac{1}{2}\|u\|^2_{\mathcal{H}^{1}(\R^d)}\Big)du,
\end{align}
again with the understanding that it properly corresponds to $\text{Law}(u^\o)$ with $u^\o$ as in \eqref{map}.
When $\ld_N\equiv\ld>0$ (focusing case)\footnote{Here, the notation ``$\ld_N \equiv \ld$''
means that the sequence $\{\ld_N\}_{N\in \N}$
is constant, taking the value $\ld$.} and $K_N\equiv K$, it has been proved that the Gibbs measure $\rho$ is normalizable if and only if $p<p^*(d):=2+\frac{4}{d}$, cf. \cite[Theorems 1.4-1.6]{RSTY}\footnote{Note that, $p^*=2+\frac{4}{d}$ satisfies the assumptions of Proposition \ref{GNS_prop}.}.
Here, by following the approach of \cite{GOTT}, we study the limit of $\rho_N$ when $\ld_N\to 0$ and $K_N\to K\in (0,\infty]$.  In particular, we show the existence of a critical threshold separating the normalizability and non-normalizability regime in correspondence of the mass critical exponent $p^*(d)$.  
\begin{theorem}
\label{main}
Let $d\ge 1$ and $p^*=2+\frac{4}{d}$. Let $\left\{\ld_N\right\}_{N\in \N}$ be a non-increasing sequence sequence of positive numbers tending to $0$ and $\left\{K_N\right\}_{N\in \N}$ be a non-decreasing sequence of positive numbers. Then, there exists $\ld^*\ge \ld_*>0$ such that the following hold:

\smallskip
\noi
\textup{(i) (weakly coupling regime).}
Suppose that 
\begin{align}
\ld_N\le \ld_* (K_N+\log N)^{-\frac{2}{d}}
\label{K1}
\end{align}

\noi
for any $N \in \N$.
Then, 
given any $r\geq 1$ we have
\begin{align*}
\sup_{N\in\N}\mathcal{Z}_{p^*\!,\,N}:=\sup_{N\in\N}\Big\|\ind_{\{|\int_{\R^d} \, :|\P_{N} u|^2: \, dx| \, \leq K_N\}} e^{ \ld_N R_{p^*\!,N}(u)}\Big\|_{L^r(\mu)}<\infty.
\end{align*}

\noi
In particular, we have
\begin{align}
\lim_{N\rightarrow\infty}
\ind_{\{|\int_{\R^d} \, :|\P_{N} u|^2: \, dx| \, \leq K_N\}}  
e^{\ld_N  R_{p^*\!,N}(u)}=
\ind_{\{|\int_{\R^d} \, :|u|^2: \, dx| \, \leq K\}}  \quad     \text{in}\ L^r(\mu), 
\label{exp2}
\end{align}

\noi
where 
$K = \lim_{N \to \infty}K_N \in (0, \infty]$.
Here, 
\[\int_{\R^d} :\!|u|^2\!:dx = \lim_{N \to \infty}\int_{\R^d} :\!|\P_{N} u|^2\!:dx,\]

\noi
where the limit is understood in 
$L^r(\mu)$ and $\mu$-almost surely
as in Lemma \ref{LEM:conv0}.
As a consequence, 
we have

\smallskip
\begin{itemize}
\item[\textup{(i.a)}] 
If $K  =  \lim_{N \to \infty}K_N  = \infty$, 
then the truncated Gibbs measure $\rho_N$
in \eqref{trunc_gibbs}
converges in total variation to the base Gaussian measure $\mu$ in \eqref{Gaussian}
as $N \to \infty$.

\smallskip
\item[\textup{(i.b)}]
If $K = \lim_{N \to \infty}K_N < \infty$, 
then the truncated Gibbs measure $\rho_N$
in \eqref{trunc_gibbs}
converges in total variation to the base Gaussian measure
with a renormalized $L^2$ cut-off\textup{:}
\begin{align*}
\ind_{\{|\int_{\R^d} \, :|u|^2: \, dx| \, \leq K\}}  d\mu, 
\end{align*}
	
\noi
as $N \to \infty$.

\end{itemize}

\smallskip
\noi
\textup{(ii) (strongly coupling regime).}
Suppose that 
\begin{align}
\ld_N\ge \ld^{*}(K_N+\log N)^{-\frac{2}{d}}
\label{K2}
\end{align}

\noi
for any sufficiently large $N \gg 1$.
Then, we have 
\begin{align}
\sup_{N\in\N}\mathcal{Z}_{p^*\!,\,N}=
\sup_{N\in\N}\Big\|\ind_{\{|\int_{\R^d} \, :|\P_{N} u|^2: \, dx| \, \leq K_N\}} e^{ \ld_N  R_{p^*\!,N}(u)}\Big\|_{L^1(\mu)}
=\infty.
\label{Gibbs3}
\end{align}

\noi
As a consequence, 
the truncated Gibbs measure $\rho_N$ in \eqref{trunc_gibbs}
does not converge   to any  limit in total variation, even up to a subsequence.

\end{theorem}

Note that, Theorem \ref{main} exhibits the same structure of \cite[Theorem 1.1]{GOTT} where the authors studied log-correlated Gibbs measures on the $d$-dimensional torus with weakly
interacting focusing quartic potentials. In particular, they show that the weakly interacting model proposed by Brydges and Slade \cite{BS} is critical when $\ld_N\sim (K_N+\log N)^{-1}$. Thus, Theorem \ref{main} shows that in the case of the focusing nonlinear
Schr\"odinger equation on $\R^d$ with harmonic potential and critical non linearity, $\ld_N\sim (K_N+\log N)^{-\frac{2}{d}}$ represents a critical behavior for the coupling constant, dividing the normalizability and the non-normalizability regime. Moreover, we see that the truncated focusing Gibbs measures \eqref{trunc_gibbs} essentialy trivialize, either converging to the base Gaussian free field on $\R^d$ (possibly with a suitable $L^2$ cut-off), or diverging. 
\section{Notations and preliminary lemmas}
Let $A\les B$ denote an estimate of the form $A\leq CB$ for some constant $C>0$. We write $A\sim B$ if $A\les B$ and $B\les A$, while $A\ll B$ denotes $A\leq c B$ for some small constant $c> 0$. 
We use $C>0$ to denote various constants, which may vary line by line.
\subsection{Deterministic estimates}
First, we recall the following general version of the Young's inequality; see  \cite[Theorem 156 on p.\,111]{HLP}. 
\begin{lemma}\label{LEM:Young}
Let $f$ be a strictly increasing function on $\R_+$
such that   $f(0) = 0$ and its inverse $f^{-1}$ is also strictly increasing.
Then, for any $a, b \ge 0$, we have
\begin{align}
ab \le \int_0^a f(x) dx   +  \int_0^b f^{-1}(x) dx 
\label{Young1}
\end{align}
	
\noi
with equality if and only if $b = f(a)$.
In particular, applying \eqref{Young1}
to  $f(x) = e^x - 1$ and $f^{-1}(x) = \log(1+x)$
\textup{(}with $b$ replaced by $b-1$\textup{)}, 
we have 
\begin{align*}
ab \le e^{a} + b \log b - b
\end{align*}
	
\noi
for any $a \ge  0$ and $b \ge 1$.
	
\end{lemma}

In what follows, we recall the following result whose  proof can be found in  \cite[Lemma 4.1]{RSTY}.

Let  $f:\R^d\to \R$ be a Schwartz function such that $\|f\|_{L^2(\R^d)}=1$ and $\ft{f}$ is supported in $\left\{\frac{1}{2}< |\xi|\le 1\right\}$. If $d\ge 2$ we further assume that $f$ is radial. Let $M>0$. We define $f_M$ by 
 \begin{align}\label{fM}
 f_M(x):=M^{\frac{d}{2}}f(Mx).
 \end{align}  
Then, the following holds:
\begin{lemma}\label{asymptotics}
Let $s\in \R$ and $f_M$ as in \eqref{fM}. Then, 
\begin{align}	
&\int_{\R^d}|f_M|^2dx=1, \label{L2scaling}\\
& \int_{\R^d} |f_M|^{p}dx \sim M^{\frac{dp}{2}-d}, \label{LPscaling}\\
& \int_{\R^d} |\mathcal{L}^{\frac{s}{2}}f_M|^{2}dx\les M^{2s},
\label{H_1}
\end{align}
for any $p>0$, $M\gg 1$. 
\end{lemma}
Next, we present a technical result which will play a role in the proof of  non-normalizability, i.e., Theorem \ref{main} (ii). Indeed, unlike the proof of \cite[Theorem 1.1]{GOTT}, $\P_Mf_M\neq f_M$ preventing us to directly apply Lemma \ref{asymptotics} in estimating the norms of the drift terms defined in \eqref{drift}. However, Lemma \ref{LEM:1} (and in particular Corollary \ref{conv_norm}) allows us to say that a similar version to Lemma \ref{asymptotics}  is valid when $f_M$ is replaced by $\P_{N}f_M$ provided $M^{\beta}\sim N$ and $\beta>\beta^*$ where $\beta^*$ is defined by 
\begin{align}\label{beta}
\beta^*:=\begin{cases} 5, & \mbox{if }d=1 \\ \frac{3d}{d-1}, & \mbox{if }d\ge 2.
\end{cases}
\end{align}
\begin{lemma}\label{LEM:1}
Let  $n\in \N$ and $M>0.$ Let $h_n$ as in \eqref{EIGEN} and $f_M$ as in \eqref{fM}. Then, for any $M\gg 1$ the following hold:
\begin{itemize}
\item[(i)] If $d=1$ then
\begin{align}
\label{d1}
|\langle f_M, h_n\rangle_{L^{2}(\R)}|\les \jb{n}^{-\frac{3}{2}} M^{\frac{5}{2}};
\end{align}
\item[(ii)] If $d\ge 2$ then
\begin{align}\label{d2}
|\langle f_M, h_n\rangle_{L^{2}(\R^d)}|\les \jb{n}^{-\frac{d+3}{4}} M^{\frac{3d}{2}}.
\end{align}
\end{itemize}
\end{lemma}
\begin{proof}
In what follows, we apply a similar strategy to \cite{B1}. In view of \eqref{fM}, it is enough to control the term 
\begin{align}\label{I_term}
I:=\int_{\R^d}f(Mx)h_n(x)dx.
\end{align}
Note that, if $n=0$, by \eqref{EIGEN} and \eqref{I_term} we immediately obtain that 
\begin{align*}
|I|\les  \|f\|_{L^{\infty}(\R^d)}.  
\end{align*}
As a consequence, in the remaining part of the proof we always assume $n\ge 1$. 
\vspace{2mm}

\noi
\textbf{Case 1:} $d=1$. 

First of all, we recall the following recursive relation
\begin{align}\label{REC}
\frac{d}{dx}H_{n+1}(x)=(2n+2)H_n(x),
\end{align}
which  follows from \cite[eq. (2.7)]{B1}.
Then, by combining \eqref{EIGEN} with \eqref{REC} and integrating by part,  we deduce 
\begin{align*}
I&=-\frac{c_n}{(2n+2)}\int_{\R}e^{-\frac{x^2}{2}}H_{n+1}(x) (Mf'(Mx)-xf(Mx))dx \notag \\
& = -\frac{1}{\sqrt{2n+2}}\int_{\R} c_{n+1} e^{-\frac{x^2}{2}}H_{n+1}(x) (Mf'(Mx)-xf(Mx))dx.
\end{align*}
Note that the boundary term arising from integrating by parts vanishes because of $f$ being rapidly decaying at infinity. 

%
\noi 
Thus, by integrating by parts two more times, we obtain
\begin{align}
|I|\les \frac{M^2}{(2n+2)^{\frac{3}{2}}}\int_{\R}|f'''(x)|dx  \les n^{-\frac{3}{2}}M^2 , \label{I_bound}  
\end{align}
provided $M$ is sufficiently large.  Note that, to derive \eqref{I_bound} we also used  the uniform bound \eqref{bound_H}.
As a result,  \eqref{d1} follows from  \eqref{I_bound},  \eqref{I_term} and \eqref{fM}. 
\vspace{2mm}

 \noi  
\textbf{Case 2:} $d\ge 2$. 

From \eqref{EIGEN}, and \eqref{L_formula},  we can recast \eqref{I_term} as follows 
\begin{align}\label{I_general}
I&=c(d)\frac{(-1)^n}{4^n\sqrt{n!\Gamma(n+\frac{d}{2})}}\sum_{m_1+\dots +m_d=n}\int_{\R^d}f(Mx) \prod_{i=1}^{d}H_{2m_i}(x_i)e^{-\frac{x_i^2}{2}}dx.
\end{align}
Next, we argue similarly to the one-dimensional case described above.  As a matter of fact, by \eqref{REC} and integrating by parts twice in $x_d$ (and recalling that $f$ has fast decay), we have
\begin{align}\label{1_it}
\int_{\R^d}&f(Mx)\prod_{i=1}^{d}H_{2m_i}(x_i)e^{-\frac{x_i^2}{2}}dx\notag\\
&=\frac{1}{4(2m_d+1)(2m_d+2)}\int_{\R^{d-1}}\prod_{i=1}^{d-1}H_{2m_i}(x_i)e^{-\frac{x_i^2}{2}}\Big(\int_{\R}H_{2m_d+2}(x_d)e^{-\frac{x^2_d}{2}}g(M,x)dx_d\Big)d\bar{x},
\end{align}
where we denoted $d\bar{x}:=\prod_{i=1}^{d-1}d{x_i}$ and 
\begin{align*}
g_M(x):=M^2\partial_{x_d}f(Mx)-f(Mx)-2Mx_d\partial_{x_d}f(Mx)-x^2_df(Mx).
\end{align*}
Note that, since $f\in \mathcal{S}(\R^d)$ we have 
\begin{align*}
\int_{\R^d}|g_M(x)|dx\les M^{2-d}.
\end{align*}
Thus, by iterating the argument in \eqref{1_it} for the remaining $d-1$ variables and \eqref{EIGEN}, we infer
\begin{align}
\int_{\R^d}f(Mx)&\prod_{i=1}^{d}H_{2m_i}(x_i)e^{-\frac{x_i^2}{2}}dx\notag\\
&=\bigg(\frac{1}{4^d}\prod_{i=1}^{d}\frac{c^{-1}_{2m_i+2}}{(2m_i+1)(2m_i+2)}\bigg)\int_{\R^d}\tilde{g}_M(x)\prod_{i=1}^{d}c_{2m_i+2}H_{2m_i+2}(x_i)e^{-\frac{x^2_i}{2}}dx\notag\\
&=\bigg(\frac{\pi^{\frac{d}{4}}}{4^d}\prod_{i=1}^{d}\frac{((2m_i+2)!)^{\frac{1}{2}}2^{m_i+1}}{(2m_i+1)(2m_i+2)}\bigg)\int_{\R^d}\tilde{g}_M(x)\prod_{i=1}^{d}h_{2m_i+2}(x_i)dx,
\label{integr_formula}
\end{align}
for some function $\tilde{g}_M$ which is Schwartz and  such that 
\begin{align}
\label{gtb}
\int_{\R^d}|\tilde{g}_M(x)|dx\les M^{d}.
\end{align} 
Hence, by combining \eqref{integr_formula} with \eqref{I_general} we have (up to multiplicative positive constant independent of $n$ and $M$)
\begin{align}
I\
\sim \mathcal{S}_{d,n}\int_{\R^d}\tilde{g}_M(x)\prod_{i=1}^{d}h_{2m_i+2}(x_i)dx, 
\label{I_d}
\end{align}
where $\mathcal{S}_{d,n}$ is defined by 
\begin{align}
\mathcal{S}_{d,n}:=\frac{(-1)^n}{2^{n}\sqrt{n!\,\Gamma(n+\frac{d}{2})}}\sum_{m_1+\dots +m_d=n}\,\prod_{i=1}^{d}\frac{((2m_i)!)^{\frac{1}{2}}}{((2m_i+1)(2m_i+2))^{\frac{1}{2}}}.
\label{S}
\end{align}
By \eqref{bound_H} and \eqref{gtb} we have that 
\begin{align}\label{M_bound}
\bigg|\int_{\R^d}\tilde{g}_M(x)\prod_{i=1}^{d}h_{2m_i+2}(x_i)dx\bigg|\les M^{d}.
\end{align}
Next, we claim that $\mathcal{S}_{d,n}$ defined in \eqref{S} satisfies the bound
\begin{align}\label{sum_prod}
|\mathcal{S}_{d,n}|\les {n^{-\frac{d+3}{4}}}.
\end{align}
To begin with, we recall that if there exists $k\in \N^{*}$ such that $\sum_{i=1}^{k}m_i=n$ where  $m_i\in \N^{*}$ for all $i\in \{1,2,\dots,k\}$, then 
\begin{align}\label{max_prod}
 \prod_{i=1}^{k}(2m_i)!\le (2(n-k+1))!.
\end{align}
Let  $n, k\in \N^{*},$ with $k\le d$. We define the set $A_{n,k}\subset \N^{d}$ as follows:  
\begin{align}
\label{defA}
A_{n,k}:=\left\{(m_1,\dots, m_d)\in \N^d: \sum_{i=1}^d m_i=n,  \# \left\{m_i\neq 0\right\}=k \right\},
\end{align}
where  by $\#A $ we denote the cardinality of the set $A$.
For the sake of notation, we also set $\overline{m}_d:=(m_1,\dots,m_d)$.

Note that, if $\overline{m}_d\in A_{k,n}$ there exists $j\in \{1,\dots,d\}$ such that $m_j\ge \frac{n}{k}$. In particular, 
\begin{align}
\label{up1}
\prod_{i=1}^{d}\frac{1}{((2m_i+1)(2m_i+2))^{\frac{1}{2}}}\les n^{-1}. 
\end{align}
Furthermore, 
\begin{align}
\label{223}
\#A_{n,k}\le  \binom{n+k-1}{n}\binom{d}{k}.
\end{align}
Thus, by \eqref{S}, \eqref{max_prod}, \eqref{defA},   \eqref{up1}, \eqref{223} and Stirling's formula for the Gamma function,  we have
\begin{equation}\label{count}
\begin{split}
|\mathcal{S}_{d,n}|&= \frac{1}{2^{n}\sqrt{n!\,\Gamma(n+\frac{d}{2})}}\sum_{k=1}^{d}\,\sum_{
\overline{m}_d\in A_{n,k}}\prod_{i=1}^{d}\frac{((2m_i)!)^{\frac{1}{2}}}{((2m_i+1)(2m_i+2))^{\frac{1}{2}}}\\
&\les \frac{1}{2^{n}\sqrt{n!\,\Gamma(n+\frac{d}{2})}}\sum_{k=1}^{d}\frac{((2(n-k+1))!)^{\frac{1}{2}}}{n}  \#A_{n,k}\\
& \les \frac{1}{2^{n}\sqrt{n!\, \Gamma(n+\frac{d}{2})}}\sum_{k=1}^{d}\frac{((2(n-k+1))!)^{\frac{1}{2}}}{n} \binom{n+k-1}{n}\\
& \les n^{-\frac{d+3}{4}}.
\end{split}
\end{equation}
 Finally, by combining \eqref{count},  \eqref{sum_prod},  \eqref{M_bound} with \eqref{I_d} we obtain \eqref{d2}.
\end{proof}
As a product of Lemma \ref{LEM:1}, we prove the following result which will be used in the proof of Theorem \ref{main} (ii).

\begin{corollary}\label{conv_norm}
Let $\beta>\beta^{*}$ where $\beta^*$ is defined in \eqref{beta} and $N\in \N$. Assume that  $M=M(N)\in \N$ is such that  $M^{\beta}\sim N$ for $N\gg 1$. Let  $f_M$ be as in \eqref{fM}.  Then, 
\begin{align*}
\lim_{N\to \infty}\|\P_{N}f_M-f_M\|_{\mathcal{H}^{1}(\R^d)}=0.
\end{align*}
The same conclusion holds when the $\mathcal{H}^{1}(\R^d)$ norm is replaced by any $L^p$ norm, where $p$ satisfies the assumptions of Proposition \ref{GNS_prop}. 
In particular, 
\begin{align}	
&\int_{\R^d}|\P_{N}f_M|^2dx\sim 1, \label{L2scaling1}\\
& \int_{\R^d} |\P_{N}f_M|^{p}dx \sim N^{\frac{1}{\beta}\big(\frac{dp}{2}-d\big)}, \label{Lpscaling1}\\
&  \int_{\R^d} |\mathcal{L}^{\frac{1}{2}}\P_{N}f_M|^{2}dx\les N^{\frac{2}{\beta}},
\label{H_11}
\end{align}
for any $p$ as in Proposition \ref{GNS_prop} and  $N\gg 1$.
\end{corollary}
\begin{proof}
We recall that, since $f_M$ is Schwartz (and radial for $d\ge 2$), we can write $f_M=\sum_{n=0}^{\infty}\langle f_M, h_n\rangle_{L^{2}(\R^d)} h_n$, and 
\begin{align}
\|f_M\|^2_{\mathcal{H}^1(\R^d)}= \sum_{n=0}^{\infty}\nu^{2}_n |\langle f_M, h_n\rangle_{L^{2}(\R^d)}|^{2},
\label{s11}
\end{align} 
where $\nu_n$ is defined by \eqref{ld_n}.
In particular,  by \eqref{proj}, \eqref{s11} and Lemma  \ref{LEM:1}
\begin{align}
\|\P_{N}f_M-f_M\|^2_{\mathcal{H}^{1}(\R)}& =\sum_{N+1}^{\infty}\nu^{2}_n |\langle f_M, h_n\rangle_{L^{2}(\R)}|^{2} \notag \\
& \les M^{5}\sum_{N+1}^{\infty} \frac{1}{n^2}\sim N^{\frac{5}{\beta}-1},
\label{h1bis}
\end{align}
which vanishes for large $N$ as long as $\beta>5$.
Similarly, if $d\ge 2$, 
\begin{align}
\label{h1}
\|\P_{N}f_M-f_M\|^2_{\mathcal{H}^{1}(\R^d)}& =\sum_{N+1}^{\infty}\nu^{2}_n |\langle f_M, h_n\rangle_{L^{2}(\R^d)}|^{2} \notag \\
& \les M^{3d}\sum_{N+1}^{\infty} {n^{-\frac{d+1}{2}}}\sim N^{\frac{3d}{\beta}-(d-1)},
\end{align}
which vanishes provided $\beta>\frac{3d}{d-1}$.  Thus, by \eqref{L2scaling}, \eqref{H_1} and  \eqref{h1} we derive \eqref{L2scaling1} and \eqref{H_11}. Furthermore, \eqref{GNS}, \eqref{h1bis} and  \eqref{h1} yield
\begin{align*}
\|\P_ {N}f_M-f_M\|^p_{L^p(\R^d)}\to 0\quad \text{as}\ N\to +\infty,
\end{align*}
proving (from \eqref{LPscaling}) \eqref{Lpscaling1}.
By combining the above analysis with Lemma \ref{asymptotics} we conclude the proof. 
\end{proof}

\subsection{Variational formulation}
In order to prove Theorem \ref{main}, we recall a
variational formula for the partition functional that has been of prominent use since the seminal work \cite{BG}
by Barashkov and Gubinelli;
see also \cite{OOT, OST, OOT2, TW}.
In this paper, we employ \cite[Lemma 2.1]{RSTY}.
Let us fix a probability space $(\Omega, \mathcal{F},\PP)$. Let $W(t)$ be  a cylindrical Brownian motion in $L^2(\R)$ (respectively $L^{2}_{\text{rad}}(\R^d))$ defined by
\begin{align*}
W(t)=\sum_{n= 0}^{\infty} B_n(t)h_n,
\end{align*}
\noi
where  
$\{B_n\}_{n \in \N}$ is a sequence of mutually independent complex-valued\footnote{By convention, we normalize $B_n$ such that $\text{Var}(B_n(t)) = t$. In particular, $B_0$ is  a standard real-valued Brownian motion.} Brownian motions and $\{h_n\}_{n\in \N}$ is the eigenbasis \eqref{EIGEN}.
Then, we define a centered Gaussian process $Y(t)$
by 
\begin{align}
Y(t)=  \mathcal{L}^{-\frac 12}W(t)=\sum_{n= 0}^{\infty}\frac{B_n(t)}{\nu_n}h_n,
\label{P2}
\end{align}
which is well defined in $L^{2}(\Omega;\mathcal{H}^{-s}(\R^d))$.
\noi
In the following, we use the shorthand notation:  $Y = Y(1)$.
Then, 
we have $\Law(Y) = \mu$, 
where $\mu$ is the  Gaussian free field defined in~\eqref{mu_def}.
Given $N \in \N$, 
we set   $Y_N = \P_{N} Y $
such that 
$\Law(Y_N) = (\P_{N})_*\mu$, 
i.e.~the pushforward of $\mu$ under 
the spectral projector $\P_{N}$ in~\eqref{proj}.

Next, let $\Ha$ denote the space of drifts, 
which are progressively measurable processes 
belonging to 
$L^2([0,1]; L^2(\R^d))$, $\PP$-almost surely. 
Then,  the  Bou\'e-Dupuis variational formula \cite{BD, Ust}
reads as follow; see \cite{Zhang} and \cite[Appendix A]{TW}
for infinite-dimensional versions.

\begin{lemma}\label{variational_f}
Let $Y(t)$ be as in \eqref{P2}. Suppose that  $F:C^\infty(\T^d) \to \R$
is measurable such that $\E\big[|F(Y_N)|^p\big] < \infty$
and $\E\big[|e^{F(Y_N)}|^q \big] < \infty$ for some $1 < p, q < \infty$ with $\frac 1p + \frac 1q = 1$.
Then, we have
\begin{align*}
\log \E\Big[e^{F(Y_N)}\Big]
= \sup_{\dr \in \mathbb H_a}\E\bigg[ F( Y_N + \P_{N}(I(\dr)(1)) - \frac{1}{2} \int_0^1 \| \dr(t) \|_{L^2_x}^2 dt \bigg], 
\end{align*}

\noi
where  
the expectation $\E = \E_\PP$
is taken  with respect to the underlying probability measure~$\PP$.
Here, $I(\dr)$ is defined by 
\begin{align}
I(\dr)(t)  =  \int_0^t \mathcal{L}^{-\frac 12} \dr(t') dt'.
\label{Dr1}
\end{align}
\end{lemma}

In particular, from \cite[Lemma 2.2]{RSTY} we recall the following:
\begin{lemma}\label{lem13} Let  $\dr\in \mathbb{H}_a$ and $I(\dr)(1)$ be as in \eqref{Dr1}. Then, we have 
\begin{align*}
\| I(\dr)(1)\|^2_{\mathcal{H}^{1}(\R^d)}\le \int_{0}^{1}\|\dr(t)\|^2_{L_x^{2}}dt.
\end{align*}
\end{lemma}

In the sequel, Lemma \ref{variational_f} will be applied to  
\begin{align*}
F(Y_N)=\frac{\ld_N}{p^*}\|Y_N\|^{p^*}_{L^{p^*}(\R^d)}\cdot \ind_{\{|\int_{\R^d} \, : |Y_N|^2: \, dx| \, \leq K_N\}}.
\end{align*}

\subsection{Stochastic tools}
In this subsection, we state several useful lemmas from
stochastic analysis. We begin by stating  the Wiener chaos estimate
(\cite[Theorem~I.22]{Simon}).
To this aim, we first recall some basic definitions the reader can find in \cite{B}.

Let $(H, B, \mu)$ be an abstract Wiener
space, where $\mu$ is a Gaussian measure on a separable Banach space B with $H \subset B$ as its
Cameron-Martin space. Let $\left\{e_j\right\}_{j\in \N}$ be a complete orthonormal system of $H^{*} = H$ such that $\left\{e_j\right\}_{j\in \N}\subset B^*$.  We define a polynomial chaos of order $k$ of the form
\begin{align*}
\prod_{j=0}^{\infty}H_{k_j}(\langle x , e_j \rangle )
\end{align*}
where $x\in B$, $k_j \neq 0$ for only finitely many $j$’s with $k=\sum_{j=0}^{\infty}k_j $, $H_{k_j}$
is the Hermite polynomial
of degree $k_j$, and $\langle \cdot, \cdot \rangle$  denotes the $B$--$B^*$ duality pairing. We then define $\mathcal{H}_k$ as the closure
of polynomial chaoses of order $k$ under $L^{2}(B, \mu)$. The elements in $H_k$ are called homogeneous
Wiener chaoses of order $k$. We also write
\begin{align*}
\mathcal{H}_{\le k}=\bigoplus_{j=0}^{k}\mathcal{H}_j,
\end{align*}
$k\in \N$. Then, we recall the following Wiener Chaos estimate:
\begin{lemma}
\label{hyper}
Let $k\in \N$. Then, we have 
\begin{align*}
\|X\|_{L^{p}(\Omega)}\le (p-1)^{\frac{k}{2}}\|X\|_{L^{2}(\Omega)},
\end{align*}
for any $p\ge 2$ and $X\in \mathcal{H}_{\le k}$. 
\end{lemma}
The proof of Lemma \ref{hyper} follows from the hypercontractivity
of the Ornstein-Uhlenbeck semigroup.
In particular, the Wick power $:\!|Y_N|^2\!:$  are homogeneous
Wiener chaos of order 2.

Next, we recall a series of results corresponding to \cite[Corollary 2.4, Corollary 2.6, Lemma 2.7, Lemma 4.2]{RSTY}. In particular, Lemma \ref{small_mu} below is fundamental to prove convergence of the indicator $\ind_{\{|\int_{\R^d} \, : |\P_{N} u|^2: \, dx| \, \leq K_N\}}$ function which in turn implies the validity of \eqref{exp2}.
\begin{corollary}\label{Lp_conv}
Let  $p$ be as in Proposition \ref{GNS_prop}. Then, $Y \in L^p(\R^d)$. Moreover,  there exists $\gamma=\gamma(d,p)>0$ such that 
\begin{align*}
	\E\Big[\|Y-Y_N\|^p_{L^{p}(\R^d)}\Big]\les_p N^{-\gamma}.
\end{align*}
\end{corollary}

\begin{corollary}\label{LEM:conv0}
Let  $r\ge 1$. Then, 
\begin{align*}
	\bigg|\!\bigg| \int_{\R^d}:\!|Y_N|^2\!:dx\bigg|\! \bigg|_{L^r(\Omega)}\les_r 1.
\end{align*}
Moreover,   $\int_{\R^d}:\!|Y_N|^2\!:dx$ is a Cauchy sequence in $L^{r}(\Omega)$ and, 
\begin{align*}
\bigg|\!\bigg| \int_{\R^d}:\!|Y_N|^2\!:dx-\int_{\R^d}:\!|Y|^2\!:dx\bigg|\! \bigg|_{L^r(\Omega)}\les_r N^{-\frac{1}{2}}.
\end{align*} 
\end{corollary}

\begin{lemma}
\label{small_mu}
Let  $\mu$ as in \eqref{mu_def}.  Then, for every $K\in \R$ and $\eps>0$,
\begin{equation*}
\mu\bigg(\int_{\R^d}:\!|\P_{N}u|^2\!:dx \in [K-\eps,K+\eps]\bigg)\les \eps,
\end{equation*}
where the implicit constant is independent of $K$ and $\eps$. 
In particular, the same holds true for $\int_{\R^d}:\!|u|^2\!:dx=\lim_{N\to +\infty}\int_{\R^d}:\!|\P_{N}u|^2\!: dx$.
\end{lemma}

\begin{lemma}\label{lem13bis}
Given $N\ge M\gg 1$, define $Z_M$ by its coefficients in the eigenfunction expansion of $\L$ as follows. For $0\le n\le M$, $\widetilde{Z}_M(n,t)$ is a solution of the following differential equation:
\begin{align*}
\begin{cases} d\widetilde{Z}_M(n,t)=\nu^{-1}_n\sqrt{M}(\widetilde{Y}_N(n,t)-\widetilde{Z}_M(n,t))dt \\ {\widetilde{Z}}_{M}|_{t=0}=0,
\end{cases}
\end{align*}
where 
\begin{align*}
\widetilde{Y}_N(n,t)=\int_{\R^d}Y_N(t,x)h_n(x)dx, 
\end{align*}
and we set ${\widetilde{Z}}_{M}\equiv 0$ for $n>M$. Then, 
\begin{align}
Z_M(t,x):=\sum_{n=0}^{M}{\widetilde{Z}}_{M}(n,t)h_n(x)
\label{Z_def}
\end{align}
is a centered Gaussian process in $L^{2}(\R^d)$, which satisfies $\P_{M}Z_M=Z_M$. Moreover, we have 
\begin{align}
&\E\big[\|Z_M\|^2_{L^{2}(\R^d)}\big]\sim \log M, \notag\\
&\E\bigg[2\textup{Re}\int_{\R^d}Y_N \overline{Z_M}dx-\int_{\R^d}|Z_M|^2dx\bigg]\sim \log M, \label{alpha_eq}\\
&\E\bigg[\Big|\!\!:\!\!\|Y_N-Z_M\|^2_{L^{2}(\R^d)}\!\!:\!\!\Big|^2 \bigg]\les M^{-1}\log M, \notag\\
&\E\bigg[\Big|\!\int_{\R^d}Y_N f_Mdx\Big|^2 \bigg]+\E\bigg[\Big|\!\int_{\R^d}Z_Mf_M dx\Big|^2 \bigg]\les M^{-\frac{3}{2}},\notag\\
& \E\bigg[\int_{0}^{1}\left\|\frac{d}{ds} Z_M(s)\right\|^2_{\mathcal{H}^1(\R^d)}ds\bigg]\les M, \notag
\end{align}
for any $N\ge M\gg 1$, where $Z_M=Z_M(1)$ and 
\begin{align*}
\!\!:\!\!\|Y_N-Z_M\|^2_{L^{2}(\R^d)}\!\!:=\|Y_N-Z_M\|^2_{L^{2}(\R^d)}-\E\big[\|Y_N-Z_M\|^2_{L^{2}(\R^d)}\big].
\end{align*}
\end{lemma}

\section{Strongly coupling regime: non-normalizability}
In this section, we prove Theorem \ref{main}--(ii). Namely, we prove the divergence \eqref{Gibbs3} under the assumption \eqref{K2}.  The argument is closely related to the one of \cite[Section 4]{RSTY} and \cite[Section 4]{GOTT} which is inspired by \cite{TW, OOT, OOT2}. 

Let $p^*=2+\frac{4}{d}$. To begin with, note that  
\begin{align*}
\E_{\mu}&\Big[\exp(\ld_N R_{p^*\!,N}(u)\cdot \ind_{\left\{|\int_{\R^d}:\,|\P_{N}u(x)|^2\,:dx|\le K_N\right\}}\Big]\notag\\
& \ge \E_{\mu}\Big[\exp\Big(\ld_N R_{p^*\!,N}(u)\cdot \ind_{\left\{|\int_{\R^d}:\,|\P_{N}u(x)|^2 \, :dx|\le K_N\right\}}\Big)\Big]-1.
\end{align*}
Hence, the divergence \eqref{Gibbs3} follows once we prove that 
\begin{align*}
\lim_{N\to +\infty}\E_{\mu}&\Big[\exp\Big(\ld_N R_{p^*\!,N}(u)\cdot \ind_{\left\{|\int_{\R^d}: \, |\P_{N}u(x)|^2 \, :dx|\le K_N\right\}}\Big)\Big]=\infty. 
\end{align*}
For the convenience of the reader, in the sequel, we set 
\begin{align}
\label{trunc_t}
\Dr=I(\dr)(1),\quad	\Dr_N=\P_{N} I(\dr)(1).
\end{align}
To this aim, by Lemma \ref{variational_f} and \eqref{P2} we infer 
\begin{align}
\log 	\E_{\mu}\Big[&\exp\Big(\ld_N R_{p^*\!,N}(u)\cdot \ind_{\left\{|\int_{\R^d}:\,|\P_{N}u(x)|^2: \,  dx|\le K_N\right\}}\Big)\Big]\notag \\
&=\sup_{\theta\in \mathbb{H}_a}\E\Big[\ld_N R_{p^*\!,N}(Y+\Theta)\cdot \ind_{\left\{|\int_{\R^d}(:\,|Y_N|^2\,:+2\text{Re}(Y_N \overline{\Theta_N})+|\Theta_N|^2) dx|\le K_N\right\}} \notag \\
& \hphantom{XX} -\frac{1}{2}\int_{0}^{1}\|\theta(t)\|^2_{L^2}dx\Big] \label{var_formula}.
\end{align}

\noi
 From now on, we follow the main main idea of \cite{OST}  to construct a sequence of drifts $\dr\in \Ha$ such that $\Dr_N$ looks
like ``$-Y_N$+ a perturbation”, where the perturbation terms are approximations of a sequence
of blow-up profiles, which are bounded in $L^2$ but have large $L^p$-norm for $p > 2$. On the other hand, as explained in \cite{GOTT}, the lack of temporal regularity of the Brownian motion prevents us to directly use $Y_N$. Then, Lemma \ref{lem13bis} (\cite[Lemma 4.2]{RSTY}) provides a suitable approximation for $Y_N$ which is sufficient for our purpose.
By following the strategy employed in \cite{GOTT}, we split the proof in two cases depending on the behavior of the sequence $\left\{K_N\right\}_{N\in\N}$. In particular, in the first case, we use the same drift as in \cite{RSTY}; see \eqref{drift}.
In the second case, a deterministic drift is sufficient to prove the divergence of the partition function \eqref{Gibbs3}, see \eqref{drift_1}.  Note also that, the divergence will be essentially a consequence of the blow up of the quantity $\ld_N\|\P_{N}f_M\|^{p^*}_{L^{p^*}}-M^2 \log M$ provided $M$ is sufficiently small compared to $N$, see Corollary \ref{conv_norm}. 
Namely, for the rest of the proof, we fix $\beta>\beta^*$ and   $M=M(N)$ such that $M^\beta\sim N$, with  $N\gg 1$ and $\beta^*$ as in \eqref{beta}.

\vspace{0.2cm}
\noi
$\bullet$ {\bf Case 1:}
$\sup_{N\in \N}(\log N)^{-1}K_N<\infty. $

In this case, from \eqref{K2} we have 
\begin{align}
\ld_N\ge C \ld^*(\log N)^{-\frac{2}{d}}.
\label{lb1}
\end{align}
Next, we define a drift $\theta^0$ of the form

\begin{align}\label{drift}
\theta^{0}(t):=\mathcal{L}^{\frac{1}{2}}\bigg(\frac{d}{dt}Z_M(t)+\sqrt{\alpha_{M,N}}f_M\bigg),
\end{align}

\noi 
where $f_M$ is defined in \eqref{fM},  $Z_M$ is defined in \eqref{Z_def} while $\alpha_{M,N}$ by 
\begin{align}
\alpha_{M,N}:=\frac{\E\bigg[2 \text{Re}\int_{\R^d} Y_N \overline{Z_M}dx-\int_{\R^d}|Z_M|^{2}dx\bigg]}{\int_{\R^d}|\P_{N}f_M|^2 dx}\sim \log N, 
\label{amn}
\end{align}   
where we also used \eqref{L2scaling1} and \eqref{alpha_eq}. 
\noi
 Also, by \eqref{drift}, \eqref{alpha_eq} and \eqref{H_1} we infer 
\begin{align}
\int_{0}^{1}\E\big[\|\theta^{0}(t)\|^2_{L^2(\R^d)}\big]dt\les N^{\frac{2}{\beta}}\log N.
\label{eq15l}
\end{align}

\noi
In what follows, for the convenience of the reader, we set \begin{align*}
\Theta^{0}_N:=\P_{N}\Theta^{0}_N=-Z_M+\sqrt{\alpha_{M,N}}\, \P_{N}f_M.
\end{align*}

\noi 
Next, similarly to \cite[Lemma 4.4]{RSTY}, (and recalling that $M=M(N))$ we can prove that there exists $N_0=N_0(K_1)\in \N$ such that  
\begin{align}\label{ineq1}
\PP \Bigg (\Big|\int_{\R^d}(:\!|Y_N|^2\!:+2\text{Re}(Y_N \overline{\Theta^{0}_N})+|\Theta^{0}_N|^2) dx\Big |\le K_N \Bigg)\ge \frac{1}{2}
\end{align}
uniformly, for all $N\ge N_0(K_1)$. Moreover, by the mean value theorem and Young's inequality we have
\begin{align}
\begin{split}
&|R_{p^*\!,N}(Y+\Theta^{0})-R_{p^*\!,N}(\sqrt{\alpha_{M,N}}f_M)|\\
 &\phantom{XXXXX} \le \eps R_{p^*\!,N}(\sqrt{\alpha_{M,N}}\, \P_{N}f_M)+C_\eps R_{p^*\!,N}(Y-Z_M),
 \end{split}
 \label{ineq2}
\end{align}
and from \cite[eq. (4.33)]{RSTY}, we have
\begin{align}\label{ineq3}
\E\big[R_{p^*\!,N}(Y-Z_M)\big]\les 1. 
\end{align}

\noi 
Then, by combining  \eqref{var_formula}, \eqref{Lpscaling1}, \eqref{amn}, \eqref{ineq1}, \eqref{ineq2}, \eqref{ineq3}, \eqref{eq15l} with  \eqref{lb1} there exist $C_1, C_2, C_3>0$ such that 

\begin{align}
&\log 	\E_{\mu} \Big[\exp\big(\ld_N R_{p^*\!,N}(u)\cdot \ind_{\left\{|\int_{\R^d} \, :|\P_{N}u(x)|^2:  \,  dx|\le K_N\right\}}\big)\Big] \notag \\
&\  \ge \E\Big[\ld_N R_{p^*\!,N}(Y+\Theta^{0})\cdot \ind_{\left\{|\int_{\R^d}(\, :|Y_N|^2:+2\text{Re}(Y_N \overline{\Theta^{0}_N})+|\Theta^{0}_N|^2) dx|\le K_N\right\}}\notag\\
& \hphantom{XX} -\frac{1}{2}\int_{0}^{1}\|\theta^{0}(t)\|^2_{L^2}dx\Big] \notag\\
&\  \ge {(1-\eps)}\ld_N R_{p^*\!,N}(\sqrt{\alpha_{M,N}}f_M)	\PP \Bigg (\Big|\int_{\R^d}(:\!|Y_N|^2\!:+2\text{Re}(Y_N \overline{\Theta^{0}_N})+|\Theta^{0}_N|^2) dx\Big |\le K_N \Bigg) \notag\\
& \hphantom{XX}-C_\eps \ld_N\E[R_{p^*\!,N}(Y-Z_M)]-\frac{1}{2}\int_{0}^{1}\E\big[\|\theta^{0}(t)\|^2_{L^2(\R^d)}\big]dt \notag\\
&\ \ge  C_1  \ld_N N^{\frac{2}{\beta}}(\log N)^{1+\frac{2}{d}}-C_2 N^{\frac{2}{\beta}}\log N-C_3\ld_N \notag  \\
&\  \ge N^{\frac{2}{\beta}}\log N (\overline{C_1} \ld^*-C_2)-C_3\ld_N\to \infty \notag ,
\end{align}
provided $\ld_*$ is sufficiently large.

\vspace{0.2cm}
\noi
$\bullet$ {\bf Case 2:} $\sup_{N\in \N}(\log N)^{-1}K_N=\infty$.

 In this case, from \eqref{K2} we obtain that 
\begin{align}\label{lower_K_big}
\ld_N\ge C\ld^* K^{-\frac{2}{d}}_N.
\end{align}

\noi
Let us define a drift $\theta^0$ of the form
\begin{align}\label{drift_1}
	\theta_\gamma(t):=\sqrt{\gamma K_N}\,\mathcal{L}^{\frac{1}{2}}f_M,
\end{align}
where $\gamma>0$ is a small constant to be chosen later. 
Then, we define $\Dr_\gamma$ as
\begin{align}
\Dr_\gamma:=I(\theta_\gamma)(1)=\sqrt{\gamma K_N}f_M.
\label{DR}
\end{align}
For the convenience of te reader, we set 
\begin{align}\label{drift_gamma}
\Dr_{N,\gamma}:=\P_{N}\Dr_\gamma.
\end{align}
Furthermore, by \eqref{drift_1} and \eqref{H_1}
\begin{align}\label{drift_K}
 \|\Dr_\gamma\|^2_{\mathcal{H}^{1}}\le \int_{0}^{1}\|\dr_\gamma(t)\|^2_{L^{2}_x}dt\les \gamma K_N M^{2}.
\end{align}
\noi
In particular, by \eqref{Wick0}, \eqref{DR}, \eqref{drift_gamma} and \eqref{Lpscaling1}  we have that 
\begin{align}
\ld_N R_{p^*\!,N}(\Dr_{\gamma})&\sim \ld_N (\gamma K_N)^{1+\frac{2}{d}}\|\P_{N} f_M\|^{p^*}_{L^{p^*}}\sim \ld_N K^{1+\frac{2}{d}}_N N^{\frac{2}{\beta}}.
\label{PK_drift}
\end{align}
Next, we claim that, by choosing $\gamma$ small enough, we have
\begin{align}\label{claim_prob}
\PP \Bigg (\Big|\int_{\R^d}(:\!|Y_N|^2\!:+2\text{Re}(Y_N \overline{\Dr_{N,\gamma}})+|\Dr_{N,\gamma}|^2) dx\Big |\le K_N \Bigg)\ge \frac{1}{2},
\end{align}
for any $N$ large.
As a result, by  \eqref{Wick0},\eqref{drift_K}, \eqref{PK_drift}, \eqref{claim_prob}, \eqref{lower_K_big}, Corollary \ref{Lp_conv} and the fact that 
\begin{align*}
|R_{p^*\!,N}(Y+\Dr_{\gamma})-R_{p^*\!,N}(\Dr_{\gamma})|\le \eps R_{p^*\!,N}(\Dr_{\gamma})+C_\eps R_{p^*\!,N}(Y), 
\end{align*}
we conclude that 
\begin{align}
\log 	\E_{\mu} &\Big[\exp\Big(\ld_N R_{p^*\!,N}(u)\cdot \ind_{\left\{|\int_{\R^d} \, :|\P_{N}u(x)|^2: \, dx|\le K_N\right\}}\Big)\Big]\notag \notag \\
& \ge \E\Big[\ld_N R_{p^*\!,N}(Y+\Dr_{\gamma})\cdot \ind_{\left\{|\int_{\R^d}(:\,|Y_N|^2\,:+2\text{Re}(Y_N \overline{\Dr_{N,\gamma}})+|\Dr_{N,\gamma}|^2) dx|\le K_N\right\}}\notag\\
& \hphantom{XX} -\frac{1}{2}\int_{0}^{1}\|\theta_\gamma(t)\|^2_{L^2}dx\Big] \notag\\
& \ge {(1-\eps)}\ld_N R_{p^*\!,N}(\Dr_{\gamma})	\PP \Bigg (\Big|\int_{\R^d}(:\!|Y_N|^2\!:+2\text{Re}(Y_N \overline{\Dr_{N,\gamma}})+|\Dr_{N,\gamma}|^2) dx\Big |\le K_N \Bigg) \notag\\
& \hphantom{XX}-C_\eps \ld_N \E[R_{p^*\!,N}(Y)]-\frac{1}{2}\int_{0}^{1}\E\big[\|\theta_\gamma(t)\|^2_{L^2(\R^d)}\big]dt \notag\\
& \ge  C_1  \ld_N N^{\frac{2}{\beta}}K_N^{1+\frac{2}{d}}-C_2K_N N^{\frac{2}{\beta}}-C_3\ld_N \notag  \\
& \ge N^{\frac{2}{\beta}}K_N (\overline{C_1} \ld^*-C_2)-C_3\ld_N\to \infty \notag ,
\end{align}
provided $\ld_*$ is large enough. 
It remains to prove \eqref{claim_prob}. From Corollary \ref{LEM:conv0} we infer 
\begin{align}
\begin{split}
\E& \bigg[\Big|\int_{\R^d} :\!{|Y_N|^2}\!: + 2 \text{Re}(Y_N \overline{\Dr_{N,\gamma}}) + |\Dr_{N,\gamma}|^2 dx\Big|^2\bigg] \\
& \les
1+  \E\bigg[\Big|\int_{\R^d}Y_N \overline{\Dr_{N,\gamma}}\, dx \Big|^2\bigg]
+\bigg(\int_{\R^d}|\Dr_{N,\gamma}|^2 \,dx\bigg)^2.\\
\end{split}
\label{B8}
\end{align}
Then, by \eqref{drift_gamma} and \eqref{L2scaling1} we obtain 
\begin{align}
\bigg(\int_{\R^d}|\Dr_{N,\gamma}|^2 \,dx\bigg)^2=\gamma^2K^2_N\|\P_{N} f_M\|^{4}_{L^2}\sim \gamma^2 K^{2}_N.
\label{term1}
\end{align}
Moreover, by \eqref{P2},  \eqref{drift_gamma} and \eqref{H_1}
\begin{align}
\E\bigg[\Big|\int_{\R^d}Y_N \overline{\Dr_{N,\gamma}}\, dx \Big|^2\bigg]&=\gamma K_N\E\bigg[\bigg|\sum_{n=0}^{N}\frac{B_n(1)}{\nu_n}\langle  f_M, h_n\rangle_{L^2}\bigg|^2\bigg]\notag \\
&=\gamma K_N\sum_{n=0}^{N}\nu^{-2}_n|\langle  f_M,h_n\rangle_{L^2}|^2\notag \\
& \le \gamma K_N \|\mathcal{L}^{-\frac{1}{2}}f_M\|^2_{L^2}\les \gamma K_N N^{-\frac{2}{\beta}}.
\label{term2}
\end{align}
In order to conclude, by combining \eqref{B8}, \eqref{term1}, \eqref{term2}, Chebyshev’s inequality with  the assumption $K_N\gtrsim \log N$, we obtain 
\begin{align*}
\PP \Bigg (\Big|\int_{\R^d}(:\!|Y_N|^2\!:+2\text{Re}(Y_N \overline{\Dr_{N,\gamma}})+|\Dr_{N,\gamma}|^2) dx\Big |\ge K_N \Bigg) \les \frac{1}{K^{2}_N}+\frac{\gamma}{K_N N^{\frac{2}{\beta}}}+\gamma^2\le \frac{1}{2},
\end{align*}
provided $\gamma$ is sufficiently small and $N$ is large.

\section{Weakly coupling regime: normalizability}
In this section, we prove the normalizability statement in Theorem \ref{main}. In particular, the proof of Theorem \ref{main} (i)  follows once we show that 
\begin{align}
	\label{bound_r}
	\sup_{N\in \N}\Big|\!\Big| \ind_{\left\{|\int_{\R^d} \, :|\P_{N}u(x)|^2: \, dx|\le K_N\right\}}\, \exp\big({{\ld_N}R_{p^*\!,N}(u)}\big)\Big|\!\Big| _{L^{r}(\mu)}< \infty.
\end{align}
Indeed, let us define $K:=\lim_{N\to +\infty}K_N$. Then, by combining Lemma \ref{small_mu}, Corollaries \ref{Lp_conv} and \ref{LEM:conv0} with the fact that $\ld_N\to 0$ as $N\to +\infty$,  we infer
\begin{align*}
\lim_{N\to +\infty}\exp(\ld_N R_{p^*\!,N}(u))&=\ind,\\
\lim_{N\to +\infty}\ind_{\left\{|\int_{\R^d} \, :|\P_{N}u(x)|^2: \, dx|\le K_N\right\}}&=\ind_{\left\{|\int_{\R^d} \, :|u(x)|^2: \, dx|\le K\right\}}.
\end{align*}
Then, by Vitali's convergence theorem, Theorem \ref{main} (i) follows once \eqref{bound_r} holds. Moreover, because of the exponential nature of the partition function, it is enough to prove \eqref{bound_r} for $r=1$, the general case being similar. Furthermore, by noting that 
\begin{align*}
&\ind_{\left\{|\int_{\R^d} \, :|\P_{N}u(x)|^2: \, dx|\le K_N\right\}}\cdot \exp\big({{\ld_N}R_{p^*\!,N}(u)}\big)\\
&\phantom{XX}\le \exp\Big(\ind_{\left\{|\int_{\R^d} \, :|\P_{N}u(x)|^2: \, dx|\le K_N\right\}}\cdot {{\ld_N}R_{p^*\!,N}(u)}\Big),
\end{align*}
the bound \eqref{bound_r} follows once we prove that 
\begin{align*}
\sup_{N\in \N}\E_{\mu}\Big[\exp\Big(\ind_{\left\{|\int_{\R^d} \, :|\P_{N}u(x)|^2 \, :dx|\le K_N\right\}}\cdot {{\ld_N}R_{p^*\!,N}(u)}\Big)\Big]<\infty. 
\end{align*}

\noi
By Lemma \ref{variational_f} and \eqref{P2} we have 
\begin{align}
&\log  \,\E_{\mu}\Big[\exp\Big(\ind_{\left\{|\int_{\R^d} \, :|\P_{N}u(x)|^2: \, dx|\le K_N\right\}}\cdot {{\ld_N}R_{p^*\!,N}(u)}\Big)\Big]\notag\\
&\ =\!\log \E\Big[\exp\Big(\ind_{\left\{|\int_{\R^d} \, :|Y_N(x)|^2: \, dx|\le K_N\right\}}\cdot {{\ld_N}R_{p^*\!,N}(Y)}\Big)\Big]\notag\\
&\ = \!\sup_{\dr \in \mathbb H_a}
\E\bigg[\ld_N R_{p^*\!,N}( Y + I(\dr)(1))\cdot \ind_{\{|\int_{\R^d} \, : |Y_N+\P_{N}I(\dr)(1)|^2: \, dx| \, \leq K_N\}} - \frac{1}{2} \int_0^1 \| \dr(t) \|_{L^2_x}^2 dt \bigg].
\label{WN}
\end{align}

\noi
Thus, similarly to \cite[Proposition 3.1]{GOTT}, we prove the following key proposition for normalizability, which aims to control the right hand side of \eqref{WN} uniformly in $N\in \N$.

\begin{proposition}
\label{norm_log}
Let $Y_N$ as in \eqref{P2} and $\Dr_N$ as in \eqref{trunc_t}.  If \eqref{K1} holds 
then, 
\begin{align}
\sup_{\dr\in \mathbb H_a}
\E\bigg[
\ld_N \| \Dr_N\|_{L^{p^*}}^{p^*}
\cdot\ind_{\{|\int_{\R^d} \, : |Y_N +\Dr_N|^2: \, dx| 
\leq K_N\}}	- \frac 1{4} \int_0^1 \| \dr(t)\|_{L^2_x} ^2 dt\bigg]\les 1.
	\label{H1}
\end{align}
\label{prop_norm}
\end{proposition}
\noi 
Note that, by Proposition \ref{norm_log}, \cite[Corollary 2.4]{RSTY},  \eqref{Wick0}, and  the inequality
\begin{align*}
	|z_1+z_2|^{p^*}\le (1+\varepsilon) |z_1|^{p^*}+C_\varepsilon|z_2|^{p^*}\quad \forall z_1,z_2\in \C,
\end{align*}
 we derive that 
\begin{align}
\textup{RHS of}\ \eqref{WN} \les 1,
\label{norma_f}
\end{align}
uniformly in $N\in \N$.
In view of \eqref{norma_f} we conclude the proof of Theorem \ref{main}-(i) (by assuming Proposition \ref{prop_norm}). 

\begin{proof}[Proof of Proposition \ref{prop_norm}]
On $A_N := 
\big\{|\int_{\R^d} :\!|Y_N +\Dr_N|^2\!: dx| 
\leq K_N\big\}$, we have 
\begin{align*}
\|\Dr_N\|_{L^2}^2  \le K_N +\int_{\R^d}\s_N(x)dx  + 2 \bigg|\int_{\R^d} Y_N \overline{\Dr_N} dx\bigg|, 
\end{align*}

\noi
where
 \begin{align*}
 \s_N(x)=\sum_{n=0}^{N}\frac{h^2_n(x)}{\nu^2_n}.
\end{align*}

\noi 
 In particular, $\int_{\R^d}\s_N(x)dx\sim \log N$. 
First, 
suppose that  we have
\begin{align}
\| \Dr_N  \|_{L^2}^2 \les K_N+\log N. 
\label{H4}
\end{align}

\noi 
Then, on $A_N$, by combining \eqref{GNS},  \eqref{K1} with \eqref{H4} we have
\begin{align}
\ld_N \| \Dr_N\|_{L^{p^*}}^{p^*}\les \ld_* (K_N+\log N)^{-\frac{2}{d}}  \|\Dr_N\|^2_{\mathcal{H}^1} \|\Dr_N\|^{\frac{4}{d}}_{L^2}\les \ld_* \|\Dr_N\|^2_{\mathcal{H}^1}.
\label{case1} 
\end{align}
Then, by combining Lemma \ref{lem13} with \eqref{case1}, we derive \eqref{H1} assuming \eqref{H4} and $\ld_*$ sufficiently small. 

Thus, it remains to consider the case 
\begin{align}
\|\Dr_N  \|_{L^2}^2 \les \bigg| \int_{\R^d } Y_N  \overline{\Dr_N}  dx \bigg|.
\label{H5}
\end{align}
Note that the following argument holds under the weaker assumption 
\begin{align}\label{upper_weak}
\ld_N\le \ld_*(\log N)^{-\frac{2}{d}}.
\end{align}
First, by  \eqref{GNS} and \eqref{upper_weak}
\begin{align}
\ld_N \| \Dr_N\|_{L^{p^*}}^{p^*}&\les \ld_* (\log N)^{-\frac{2}{d}}  \|\Dr_N\|^2_{\mathcal{H}^1} \|\Dr_N\|^{\frac{4}{d}}_{L^2}\notag\\
&=\big(\ld^{\frac{d}{2}}_*(\log N)^{-1}\|\Dr_N\|^{d}_{\mathcal{H}^{1}}\|\Dr_N\|^{2}_{L^2})\big)^{\frac{2}{d}}\notag\\
&=\big(\ld^{\frac{d}{2}}_*(\log N)^{-1}\|\Dr_N\|^{2}_{\mathcal{H}^{1}}\|\Dr_N\|^{2}_{L^2})\big)^{\frac{2}{d}}\|\Dr_N\|^{2-\frac{4}{d}}_{\mathcal{H}^{1}}.
\label{ineq8}
\end{align}
From Lemma \ref{LEM:Young} we infer that 
\begin{align}
\ld^{\frac{d}{2}}_*&(\log N)^{-1}\|\Dr_N||^2_{\mathcal{H}^{1}}\|\Dr_N\|^2_{L^{2}}=\ld^{\frac{d}{4}}_*(\log N)^{-1}\|\Dr_N\|^2_{L^{2}} \bigg(\ld^{\frac{d}{4}}_*\|\Dr_N\|^2_{L^2}\frac{\|\Dr_N\|^2_{\mathcal{H}^1}}{\|\Dr_N\|^2_{L^{2}}}\bigg)\notag\\
& \le \ld^{\frac{d}{4}}_*(\log N)^{-1}\|\Dr_N\|^2_{L^{2}}\bigg(e^{\ld^{\frac{d}{4}}_*\|\Dr_N\|^2_{L^2}}+\frac{\|\Dr_N\|^2_{\mathcal{H}^{1}}}{\|\Dr_N\|^2_{L^2}}\log\bigg(\frac{\|\Dr_N\|^2_{\mathcal{H}^1}}{\|\Dr_N\|^2_{L^{2}}}\bigg)\bigg)\notag\\
& \les (\log N)^{-1}e^{2\ld^{\frac{d}{4}}_*\|\Dr_N\|^2_{L^{2}}}+\ld^{\frac{d}{4}}_*\|\Dr_N\|^2_{\mathcal{H}^{1}},
\label{ineq_9}
\end{align}
where to derive \eqref{ineq_9} we have also used the definition of $\nu_n$ in \eqref{ld_n} and  
\begin{align*}
	\|\Dr_N\|^2_{\mathcal{H}^1}=\sum_{n=0}^{N}\nu^{2}_n |\langle \Dr_N, h_n\rangle_{L^2}|^2\les N\|\Dr_N\|^2_{L^2}.
\end{align*}

\noi 
Next, we claim that there exist  non-negative random variables $X_N(\omega)$, 
\begin{align}
\sup_{N \in \N} \E[ X^2_N(\o)] < \infty
\label{H6a}
\end{align}
such that 
\begin{align}
e^{ 2\ld_*^\frac{1}{4}  \|\Dr_N\|_{L^2}^2}
\les   1+   \| \Dr_N  \|_{\mathcal{H}^1 }^2
+ X_N(\o).
\label{H7}
\end{align}
provided $\ld_*$ is small enough. Assuming that \eqref{H6a} and \eqref{H7} hold, from \eqref{ineq8} and \eqref{ineq_9} we have 
\begin{align}
&\ld_N \| \Dr_N\|_{L^{p^*}}^{p^*}\les \big((\log N)^{-1}e^{2\ld^{\frac{d}{4}}_*\|\Dr_N\|^2_{L^{2}}}+\ld^{\frac{d}{4}}_*\|\Dr_N\|^2_{\mathcal{H}^{1}}\big)^{\frac{2}{d}}\|\Dr_N\|^{2-\frac{4}{d}}_{\mathcal{H}^{1}}\notag\\
& \les \big((\log N)^{-1}+(\log N)^{-1}\|\Dr_N\|^2_{\mathcal{H}^{1}}+(\log N)^{-1} X_N(\o)+\ld^{\frac{d}{4}}_*\|\Dr_N\|^2_{\mathcal{H}^{1}}\big)^{\frac{2}{d}}\|\Dr_N\|^{2-\frac{4}{d}}_{\mathcal{H}^{1}}\notag \\
& \les (\log N)^{-\frac{2}{d}}\|\Dr_N\|^{2-\frac{4}{d}}_{\mathcal{H}^{1}}+\ld^{\frac{1}{2}}_*\|\Dr_N\|^2_{\mathcal{H}^{1}}+(\log N)^{-\frac{2}{d}} (X_N(\o))^{\frac{2}{d}}\|\Dr_N\|^{2-\frac{4}{d}}_{\mathcal{H}^{1}}.
\label{upper_9}
\end{align}
Note that we can reduce to study the case $\|\Dr_N\|_{L^2}\ge 1$. Indeed, if not,  we are in the case described in \eqref{H4} which has been already analyzed. In particular, if $d>2$, \eqref{upper_9} yields
\begin{align}\label{ineq_dbig}
\ld_N \| \Dr_N\|_{L^{p^*}}^{p^*}\les \ld^{\frac{1}{2}}_*\|\Dr_N\|^2_{\mathcal{H}^{1}}+(\log N)^{-\frac{2}{d}} (X_N(\o))^{\frac{2}{d}}\|\Dr_N\|^{2-\frac{4}{d}}_{\mathcal{H}^{1}}.
\end{align}
Furthermore, if we call $B_N:=\left\{\o\in \Omega: 1\le \|\Dr_N  \|_{L^2}^2 \les \bigg| \int_{\R^d } Y_N  \overline{\Dr_N}  dx \bigg|\right\}$, by H\"older  inequality and \eqref{H6a} we have
\begin{align}\label{holder}
\E\Big[(X_N(\o))^{\frac{2}{d}}\|\Dr_N\|^{2-\frac{4}{d}}_{\mathcal{H}^{1}}\cdot\ind_{A_N\cap B_N}\Big]&\le \big(\E[X_N(\o)]\big)^{\frac{2}{d}}\big(\E[\|\Dr_N\|^2_{\mathcal{H}^{1}}\cdot\ind_{A_N\cap B_N}]\big)^{1-\frac{2}{d}}\notag\\
& \les \big(\E[\|\Dr_N\|^2_{\mathcal{H}^{1}}\cdot\ind_{A_N\cap B_N}]\big)^{1-\frac{2}{d}}.
\end{align}
Thus, from \eqref{ineq_dbig}, \eqref{holder} and Lemma \ref{lem13} we have
\begin{align*}
\E&\bigg[
\ld_N \| \Dr_N\|_{L^{p^*}}^{p^*}
\cdot\ind_{A_N\cap B_N}	- \frac 1{2} \int_0^1 \| \dr(t)\|_{L^2_x} ^2 dt\bigg]\notag\\
&\le \E\bigg[
\bigg(C\ld^{\frac{1}{2}}_*-\frac{1}{4}\bigg) \| \Dr_N\|^2_{\mathcal{H}^{1}}\cdot\ind_{A_N\cap B_N}
\bigg]+C(\log N)^{-\frac{2}{d}}\E\bigg[(X_N(\o))^{\frac{2}{d}}\|\Dr_N\|^{2-\frac{4}{d}}_{\mathcal{H}^{1}}\cdot\ind_{A_N\cap B_N}\bigg]\notag\\
& \le \bigg(C\ld^{\frac{1}{2}}_*-\frac{1}{4}\bigg)\E\bigg[\| \Dr_N\|^2_{\mathcal{H}^{1}}\cdot\ind_{A_N\cap B_N}
\bigg]+C_1(\log N)^{-\frac{2}{d}}\bigg(\E\bigg[\|\Dr_N\|^2_{\mathcal{H}^1}\cdot\ind_{A_N\cap B_N}\bigg]\bigg)^{1-\frac{2}{d}}\notag\\
& \les 1,
\end{align*}
provided $\ld_*$ is small enough. 

Next,  if $d=2$, \eqref{upper_9} leads to 
\begin{align*}
\ld_N \| \Dr_N\|_{L^{p^*}}^{p^*}\les (\log N)^{-1}+(\log N)^{-1}X_N(\o)+\ld^{\frac{1}{2}}_*\|\Dr_N\|^2_{\mathcal{H}^{1}},
\end{align*}
which (in view of \eqref{H6a}) implies \eqref{H1}.
Finally, if $d=1$ (and $\o\in B_N$), \eqref{upper_9} leads to 
\begin{align*}
\ld_N \| \Dr_N\|_{L^{p^*}}^{p^*}& \les (\log N)^{-2}+(\log N)^{-2}\|\Dr_N\|^{-2}_{L^{2}}(X_N(\o))^{2}+\ld^{\frac{1}{2}}_*\|\Dr_N\|^2_{\mathcal{H}^{1}}\notag\\
& \les (\log N)^{-2}+(\log N)^{-2}(X_N(\o))^{2}+\ld^{\frac{1}{2}}_*\|\Dr_N\|^2_{\mathcal{H}^{1}}.
\end{align*}
Then, by \eqref{H6a} and \eqref{H7} we again conclude \eqref{H1}.
In view of the above analysis, in order to finish the proof of \eqref{H1},  it remains to prove \eqref{H6a} and \eqref{H7}. We argue as in \cite{GOTT}. 

	Define the  sharp daydic frequency projections $\{\Pi_j\}_{j \in \N}$
by setting $\Pi_1 = \P_{2}$
and $\Pi_j = \P_{2^j} - \P_{2^{j-1}}$.
Then,
write  $\Dr_N $ as 
\begin{align*}
\Dr_N   = \sum_{j=1}^\infty (\al_j \proj_j Y_N  + w_j),
\end{align*}
where
\begin{align*}
\al_j &:=
\begin{cases}
\frac{\jb{\Dr_N , \proj_j Y_N }}{\|\proj_j Y_N \|_{L^2}^2},  & \text{if } \| \proj_j Y_N  \|_{L^2} \neq 0, \\
0, & \text{otherwise},
\end{cases}
\qquad
\text{and}
\qquad  
w_j :=
\proj_j \Dr_N  - \al_j \proj_j Y_N .
\end{align*}

\noi
Noting that $w_j$ is orthogonal to $\proj_j Y_N $ and $Y_N $ in $L^2(\R^d)$, 
we have 
\begin{align}
\| \Dr_N  \|_{L^2}^2
&= \sum_{j=1}^\infty \Big( |\al_j|^2 \| \proj_j Y_N  \|_{L^2}^2 + \| w_j \|_{L^2}^2 \Big), \label{HH2bis} \\
\int_{\R^d} Y_N  \overline{\Dr_N}  dx
&= \sum_{j=1}^\infty \overline{\al_j} \| \proj_j Y_N  \|_{L^2}^2.
\label{HH2}
\end{align}

Then, by \eqref{H5}, \eqref{HH2bis} and \eqref{HH2} we obtain
\begin{align}
\begin{split}
\sum_{j=1}^\infty  |\al_j|^2 \| \proj_j Y_N  \|_{L^2}^2\les \bigg| \sum_{j=1}^\infty \overline{\al_j} \| \proj_j Y_N  \|_{L^2}^2 \bigg|.
\end{split}
\label{HH3}
\end{align}

Furthemore, if we 
fix  a random number $j_0 \in \N$ (to be chosen later), we have
\begin{align}
\bigg|\sum_{j=j_0+1}^\infty \overline{\al_j} \| \proj_j Y_N  \|_{L^2}^2 \bigg|&\le \bigg(\sum_{j=1}^{\infty}2^{j}\|\proj_j \Dr_N\|^2_{L^2} \bigg)^{\frac{1}{2}}\bigg(\sum_{j=j_0+1}^{\infty}2^{-j}\|\proj_j Y_N\|^2_{L^2}\bigg)^{\frac{1}{2}}\notag\\
& \sim \|\Dr_N\|_{\mathcal{H}^{1}}\bigg(\sum_{j=j_0+1}^{\infty}2^{-j}\|\proj_j Y_N\|^2_{L^2}\bigg)^{\frac{1}{2}}.
\label{H_norm_p}
\end{align}
Similarly, by Cauchy-Schwarz, Young's inequality and \eqref{HH3} we infer 
\begin{align}
\bigg|\sum_{j=1}^{j_0} \overline{\al_j} \| \proj_j Y_N  \|_{L^2}^2 \bigg|&\le \bigg(\sum_{j=1}^{\infty}|\alpha_j|^{2}\|\proj_j Y_N\|^2_{L^2} \bigg)^{\frac{1}{2}}\bigg(\sum_{j=1}^{j_0}\|\proj_j Y_N\|^2_{L^2}\bigg)^{\frac{1}{2}}\notag\\
& \le \bigg |\sum_{j=1}^\infty \overline{\al_j} \| \proj_j Y_N  \|_{L^2}^2 \bigg|^{\frac{1}{2}}\bigg(\sum_{j=1}^{j_0}\|\proj_j Y_N\|^2_{L^2}\bigg)^{\frac{1}{2}}\notag\\
& \le \frac{1}{2}\bigg |\sum_{j=1}^\infty \overline{\al_j} \| \proj_j Y_N  \|_{L^2}^2 \bigg|+\frac{1}{2}\sum_{j=1}^{j_0}\|\proj_j Y_N\|^2_{L^2}.
\label{ineq6}
\end{align}
From \eqref{ineq6} and \eqref{H_norm_p} we deduce 
\begin{align}
\bigg|\sum_{j=1}^\infty \overline{\al_j} \| \proj_j Y_N  \|_{L^2}^2 \bigg|\les \|\Dr_N\|_{\mathcal{H}^{1}}\bigg(\sum_{j=j_0+1}^{\infty}2^{-j}\|\proj_j Y_N\|^2_{L^2}\bigg)^{\frac{1}{2}}+\sum_{j=1}^{j_0}\|\proj_j Y_N\|^2_{L^2}.
\label{eq21}
\end{align}
Next, we split as follows
\begin{align}
\sum_{j=j_0+1}^{\infty}2^{-j}\|\proj_j Y_N\|^2_{L^2}=\sum_{j=j_0+1}^{\infty}2^{-j}\int_{\R^d}:\!{|\proj_j Y_N|}^2\!:dx+\sum_{j=j_0+1}^{\infty}2^{-j}\int_{\R^d}\E\big[|\proj_j Y_N|^2\big]dx.
\label{splitting21}
\end{align}
Note that, 
\begin{align}
\int_{\R^d}\E\big[|\proj_j Y_N|^2\big]dx\le\sum_{2^{j-1}+1}^{2^j} \frac{1}{\nu^2_n}\sim \sum_{2^{j-1}+1}^{2^j} \frac{1}{n}\les 1.
\label{sum_j}
\end{align}
\noi
Now, by recalling that $\E\big[(|g_{n_1}(\o)|^2-1)(|g_{n_2}(\o)|^2-1)\big]=2\delta_{n_1,n_2}$ (cf \cite[Proof of Corollary 2.6]{RSTY}), we have that
\begin{equation*}
\E\bigg[\bigg(\sum_{n_1=2^{j-1}+1}^{2^j}\frac{|g_{n_1}(\o)|^2-1}{\nu^2_{n_1}}\bigg)\bigg(\sum_{n_1=2^{k-1}+1}^{2^k}\frac{|g_{n_2}(\o)|^2-1}{\nu^2_{n_2}}\bigg)\bigg]=\sum_{2^{j-1}+1}^{2^{j}}\frac{2}{\nu^{4}_n}\sim \sum_{2^{j-1}+1}^{2^{j}}\frac{1}{n^2}.
\end{equation*}
As a consequence, 
\begin{align}
\E\bigg [\bigg(\sum_{j=j_0+1}^{\infty}2^{-j}\int_{\R^d}:\!{|\proj_j Y_N|}^2\!:dx\bigg)^2\bigg]\les  \sum_{j=j_0+1}^{\infty} 2^{-2j} \bigg(\sum_{2^{j-1}+1}^{2^j}\frac{1}{n^2}\bigg)\sim   2^{-3j_0}.
\label{sum_large}
\end{align}
Now, let us set a random variable defined as follows:
\begin{align}
B_{1,N}(\o):=\bigg(\sum_{k=1}^{\infty}2^{\frac{5}{2}k}\Big(\sum_{j=k+1}^{\infty}2^{-j}\int_{\R^d}:\!{|\proj_j Y_N|}^2\!:dx\Big)^2\bigg)^{\frac{1}{2}}.	
\label{def_B1}
\end{align}
Then, by \eqref{sum_large} and \eqref{def_B1} we infer
\begin{align*}
\sum_{j=j_0+1}^{\infty}2^{-j}\|\proj_j Y_N\|^2_{L^2} \les 2^{-\frac{5}{4}j_0} B_{1,N}(\o)+2^{-j_0}.
\end{align*}
Note that \eqref{sum_j} also implies
\begin{align}
\sum_{j=1}^{j_0}\int_{\R^d}\E\big[|\proj_j Y_N|^2\big]dx\les j_0.
\label{sum21}
\end{align}
Next,  we set a non negative random variable $B_{2,N}(\o)$ by 
\begin{align}
\begin{split}
B_{2, N}(\o) & = 
\sum_{k = 1}^\infty
\bigg|  \int_{\R^d} :\! | \proj_{ k}Y_N |^2 \!: dx\bigg|.
\end{split}
\label{HH12}
\end{align}
Thus, by combining \eqref{splitting21}, \eqref{sum21} with \eqref{HH12} and arguing as above we infer
\begin{align}
\sum_{j=1}^{j_0}\|\proj_j Y_N\|^2_{L^2}\les B_{2,N}(\o)+j_0.
\label{eq22}
\end{align}
Finally, let us choose $j_0(\o)$ such that 
\begin{align}
2^{\frac{j_0(\o)}{2}}\sim 2+\|\Dr_N(\o)\|_{\mathcal{H}^1}. 
\label{exp_j}
\end{align} 
Then, by \eqref{H5}, \eqref{eq21}, \eqref{eq22} and \eqref{exp_j} we infer
\begin{align}
\begin{split}
\|\Dr_N\|^2_{L^{2}}&\les \|\Dr_N\|_{\mathcal{H}^{1}}\Big(2^{-\frac{5}{8}j_0}B^{\frac{1}{2}}_{1,N}(\o)+2^{-\frac{j_0}{2}}\Big)+B_{2,N}(\o)+j_0 \notag\\
& \les \log(2+\|	\Dr_N\|_{\mathcal{H}^1})+B^{\frac{1}{2}}_{1,N}(\o)+B_{2,N}(\o),\notag
\end{split}
\end{align}
and,  by arguing as in \cite[Proposition 3.1, eq. (3.27)-(3.28)]{GOTT} the non-negative random variable $X_N(\o)$ defined by
\begin{align*}
X_N(\o) = e^{\epsilon  B_{1, N}^\frac 12 (\o)
+ \epsilon B_{2, N}(\o)},
\end{align*}
satisfies \eqref{H6a} and \eqref{H7} provided  $\epsilon=\epsilon(\ld_*)$ is sufficiently small. This concludes the proof of the proposition. 
\end{proof}


\begin{ackno}\rm
D.G. was supported by the European Research Council (grant no.~864138 ``SingStochDispDyn") and
Y.W. was supported by the EPSRC New Investigator
Award (grant no. EP/V003178/1). Y.W. was also supported by the EPSRC Mathematical Sciences Small Grant (grant no. UKRI1116). The authors are very thankful to Prof. Tadahiro Oh for suggesting the problem.
\end{ackno}






\begin{thebibliography}{99}


\bibitem{BG}
N.~Barashkov, M.~Gubinelli, 
{\it  A variational method for $\Phi^4_3$}, 
Duke Math. J.
169 (2020), no. 17, 3339--3415.

\bibitem{BG2}
N. Barashkov, M. Gubinelli, {\it On the variational method for Euclidean quantum fields in infinite volume}, Probab. Math. Phys. {\bf 4} (2023), no.~4, 761--801.

\bibitem{Baras}
N. Barashkov, P. Laarne, {\it Invariance of $\Phi^4$ measure under nonlinear wave and Schr\"odinger equations on the plane}, arXiv:2211.16111 [math.AP].


\bibitem{BFV14} 
J.~Bellazzini, R.L.~Frank,  N.~Visciglia, {\it 
	Maximizers for Gagliardo-Nirenberg inequalities and related
	non-local problems}, Math. Ann.  360 (2014),
no. 3-4, 653--673.	


\bibitem{BBDBG}
P.B. Blakie, A.S. Bradley, M.J. Davis, R.J. Ballagh, C.W. Gardiner, {\it Dynamics and statistical mechanics of ultra-cold Bose gases using c-field techniques}, Adv. in Phys., 57 (2008), pp. 363--455.
	
	
\bibitem{B}
 V.~Bogachev,
 {\it Gaussian measures, Mathematical Surveys and Monographs}, 62. American Mathematical Society, Providence, RI, 1998. xii+433 pp.
 
 \bibitem{B1}
 J.P.~Boyd, 
{\it Asymptotic coefficients of hermite function series}, 
J. Math. Phys., 54 (1984), no. 3, 382--410.
 
 
\bibitem{BD}
M.~Bou\'e, P.~Dupuis,
{\it A variational representation for certain functionals of Brownian motion},
Ann. Probab. 26 (1998), no. 4, 1641--1659.


\bibitem{BBulut} 
J.~Bourgain,  A.~Bulut, 
{\it Almost sure global well posedness for the radial nonlinear Schr\"odinger equation on the unit ball I: the 2D case.,}
Ann. Inst. H. Poincar\'e Anal. Non Lin\'eaire 31 (2014), no. 6, 1267--1288.

\bibitem{BO94}
J.~Bourgain, 
{\it Periodic nonlinear Schr\"odinger equation and invariant measures}, 
Comm. Math. Phys. 166 (1994), no. 1, 1--26.
	
\bibitem{BO96}
J.~Bourgain, 
{\it Invariant measures for the 2D-defocusing nonlinear Schr\"odinger equation}, Comm. Math. Phys. 176 (1996), no. 2, 421--445. 

\bibitem{BO97}
J.~Bourgain, 
{\it Invariant measures for the Gross-Piatevskii equation,} 
J. Math. Pures Appl. 76 (1997), no. 8, 649--702. 





\bibitem{BS}
D.C.~Brydges, G.~Slade,  
{\it Statistical mechanics of the 2-dimensional focusing nonlinear Schr\"odinger equation}, Comm.Math. Phys. 182, (1996), 485--504. 

\bibitem{BTT}
N.~Burq, L.~Thomann, N.~Tzvetkov, 
{\it Long time dynamics for the one dimensional non linear Schr\"odinger equation,}
 Ann. Inst. Fourier (Grenoble) 63 (2013), no. 6, 2137--2198. 
 
 \bibitem{C1}
 E.~Carlen, J.~Fr\"ohlich, J.~Lebowitz, 
 {\it Exponential relaxation to equilibrium for a one-dimensional focusing non-linear Schr\"odinger equation with noise}, Comm. Math. Phys. 342 (2016), no. 1, 303--332.


\bibitem{C}
L.~Carlitz, 
{\it The relationship of the Hermite to the Laguerre polynomials}, 
Boll. Un. Mat. Ital., Serie 3, 16 (1961), no. 4, 386--390.

\bibitem{D1}
Y.~Deng, 
{\it Two-dimensional nonlinear Schr\"odinger equation with random radial data},
 Anal. PDE 5 (2012), no. 5, 913--960.
 
 \bibitem{D3}
 Y. ~Deng, A.~Nahmod, H.~Yue, 
 {\it Invariant Gibbs measures and global strong solutions for nonlinear Schr\"odinger equations in dimension two}, 
 Ann. of Math. 200 (2024), no. 2, 399--486

\bibitem{D2}
Y.~Deng, A.~Nahmod, H.~Yue, 
{\it Random tensors, propagation of randomness, and nonlinear dispersive equations}, Invent. math. 228 (2022), no. 2, 539--686.

\bibitem{DS}
R. Duine, H. Stoof, {\it Stochastic dynamics of a trapped Bose-Einstein condensate}, Phys. Rev. A, 65 (2001) 013603.

\bibitem{FOW}
J.~Forlano, T.~Oh, Y.~Wang,
{\it  Invariant Gibbs dynamics for the nonlinear Schr\"odinger equations on the disc}, 
preprint. 


 
\bibitem{Frank}
R.~Frank, E.~Lenzmann, L.~Silvestre,
{\it Uniqueness of radial solutions for the fractional Laplacian},
Comm. Pure Appl. Math. 69 (2016), no. 9, 1671--1726. 
 
 

\bibitem{GOTT}
D. Greco, T. Oh, L. Tao, L. Tolomeo, {\it Critical threshold for weakly interacting log-correlated focusing Gibbs measures}, Proc. Amer. Math. Soc. Ser. B {\bf 12} (2025), 150--165.


\bibitem{GLLOW}
D.~Greco, G.~Li, R.~Liang, T.~Oh, Y.~Wang, 
{\it Optimal divergence rate of the focusing Gibbs measures}, arXiv:2310.08783 [math.PR].


\bibitem{GH}
M. Gubinelli and M. Hofmanov\'a, {\it A PDE construction of the Euclidean $\phi_3^4$ quantum field theory}, Comm. Math. Phys. {\bf 384} (2021), no.~1, 1--75


\bibitem{HLP}
G.H.~Hardy, J.E.~Littlewood, G.~P\'olya, 
{\it Inequalities}, Reprint of the 1952 edition. Cambridge Mathematical Library. Cambridge University Press, Cambridge, 1988. xii+324 pp.


\bibitem{RDL}
R.~Imekraz, D.~Robert, L.~Thomann, 
{\it On random Hermite series}, Trans. Amer. Math. Soc. 368 (2016), no. 4, 2763-2792. 

\bibitem{K}
T.~Koornwinder, 
{\it The Addition Formula for Laguerre Polynomials}, 
SIAM J. Math. Anal. 8 (1977), no. 3, 535--540.


\bibitem{LW22} 
R.~Liang,  Y.~Wang, 
{\it Gibbs measure for the focusing fractional NLS on the torus}, 
SIAM. J. Math. Anal. 54 (2022), no. 6, 6096--6118.


\bibitem{LMW}
J.~Lebowitz, P.~Mounaix, W.-M.~Wang, 
{\it Approach to equilibrium for the stochastic NLS,} Comm. Math. Phys. 321 (2013), no. 1, 69--84.



\bibitem{LRS}
J.~Lebowitz, H.~Rose, E.~Speer, 
{\it Statistical mechanics of the nonlinear Schr\"odinger equation}, J. Statist. Phys.
50 (1988), no. 3-4, 657–687.



\bibitem{Nagy}
B.V.Sz.~Nagy, 
{\it \"Uber Integralgleichungen zwischen einer Funktion und ihrer Ableitung,}
 Acta Univ. Szeged. Sect. Sci. Math
  10 (1941), 64--74. 



\bibitem{OOT}
T.~Oh, M.~Okamoto, L.~Tolomeo, 
{\it Focusing $\Phi^4_3$-model with a Hartree-type nonlinearity},
 Mem. Amer. Math. Soc. 304 (2024), no. 1529, vi+143 pp. 


\bibitem{OOT2}
T.~Oh, M.~Okamoto, L.~Tolomeo, 
{\it Stochastic quantization of the $\Phi^3_3$-model}, 
Mem. Eur. Math. Soc., 16. EMS Press, Berlin, 2025, viii+145 pp. 


\bibitem{OQV}
T.~Oh, J.~Quastel, B.~Valk\'o,
{\it  Interpolation of Gibbs measures and white noise for Hamiltonian PDE}, J. Math. Pures Appl. 97 (2012), no. 4, 391--410. 

\bibitem{OST}
T.~Oh, K.~Seong, L.~Tolomeo, 
{\it A remark on Gibbs measures with log-correlated Gaussian fields}, 
Forum Math. Sigma 12 (2024), Paper No. e50. 

\bibitem{OST2}
T.~Oh,    P.~Sosoe,   L.~Tolomeo,  
{\it Optimal integrability threshold for Gibbs measures associated with focusing NLS 
on the torus,}
Invent. Math. 227 (2022), no. 3, 1323--1429.


\bibitem{RSTY}
T.~Robert, K.~Seong, L.~Tolomeo, Y.~ Wang, 
{\it Focusing Gibbs measure with harmonic potential}, 
Ann. Inst. Henri Poincar\'e Probab. Stat. 61 (2025), no. 1, 571--598. 


\bibitem{Simon}
B.~Simon, 
{\it  The $P(\varphi)_2$ Euclidean (quantum) field theory,} Princeton Series in Physics. Princeton University Press, Princeton, N.J., 1974. xx+392 pp.



\bibitem{TW}
L.~Tolomeo, H.~Weber, 
{\it Phase transition for invariant measures of the focusing Schr\"odinger equation}, 
arXiv:2306.07697 [math.AP].


\bibitem{Tzv1} 
N.~Tzvetkov, 
{\it Invariant measures for the nonlinear Schr\"odinger equation on the disc}, 
 Dyn. Partial Differ. Equ. 3 (2006), no. 2, 111--160.
 
 \bibitem{Tzv2} 
N.~Tzvetkov, 
{\it Invariant measures for the defocusing nonlinear Schr\"odinger equation}, 
 Ann. Inst. Fourier (Grenoble) 58 (2008), no. 7, 2543--2604.



\bibitem{Ust}
A.~\"Ust\"unel,
{\it Variational calculation of Laplace transforms via entropy on Wiener space and applications},
J. Funct. Anal. 267 (2014), no. 8, 3058--3083.

\bibitem{Weinstein}
M.~Weinstein,
{\it Nonlinear Schr\"dinger equations and sharp interpolation estimates},
Comm. Math. Phys. 87 (1982/83), no. 4, 567--576. 



\bibitem{Zhang}
X.~Zhang, 
{\it A variational representation for random functionals on abstract Wiener spaces},
J. Math. Kyoto Univ. 49 (2009), no. 3, 475--490. 



\end{thebibliography}
\end{document}